\documentclass[12pt, reqno]{amsart}

\usepackage{amssymb, amsmath, amsthm,stmaryrd}
\usepackage[backref]{hyperref}
\usepackage[alphabetic,backrefs,lite]{amsrefs}
\usepackage{amscd}   
\usepackage{fullpage}
\usepackage[all,cmtip]{xy} 
\usepackage{graphicx}
\usepackage{mathrsfs}
\usepackage{multirow}
\DeclareFontEncoding{OT2}{}{} 

\usepackage[usenames,dvipsnames]{color}

\newtheorem{lemma}{Lemma}[section]
\newtheorem{theorem}[lemma]{Theorem}
\newtheorem{proposition}[lemma]{Proposition}
\newtheorem{prop}[lemma]{Proposition}
\newtheorem{cor}[lemma]{Corollary}

\newtheorem{claim*}{Claim}
\newtheorem{thm}[lemma]{Theorem}
\newtheorem{defn}[lemma]{Definition}

\newtheorem*{thm_nonum}{Theorem}

\theoremstyle{definition}
\newtheorem{remark}[lemma]{Remark}

\newtheorem{example}[lemma]{Example}

\def\st{\,|\,}

\def\e{\varepsilon}
\def\O{\mathcal{O}}
\def\P{\mathbb{P}}
\def\L{\mathcal{L}}

\DeclareMathOperator{\Bl}{Bl}

\newcommand{\A}{{\mathbb A}}
\newcommand{\G}{{\mathbb G}}

\newcommand{\PP}{{\mathbb P}}
\newcommand{\C}{{\mathbb C}}
\newcommand{\F}{{\mathbb F}}
\newcommand{\Q}{{\mathbb Q}}

\newcommand{\Z}{{\mathbb Z}}

\newcommand{\kk}{{\mathbf k}}

\newcommand{\calA}{{\mathcal A}}
\newcommand{\calB}{{\mathcal B}}

\newcommand{\calE}{{\mathcal E}}

\newcommand{\calL}{{\mathcal L}}

\newcommand{\calO}{{\mathcal O}}

\newcommand{\OO}{{\mathcal O}}

\newcommand{\mm}{{\mathfrak m}}

\DeclareMathOperator{\HH}{H}
\DeclareMathOperator{\hh}{h}

\DeclareMathOperator{\Tr}{Tr}

\DeclareMathOperator{\Char}{char}

\DeclareMathOperator{\im}{im}

\DeclareMathOperator{\End}{End}

\DeclareMathOperator{\Hom}{Hom}

\DeclareMathOperator{\Aut}{Aut}
\DeclareMathOperator{\Gal}{Gal}

\DeclareMathOperator{\Norm}{Norm}
\DeclareMathOperator{\Br}{Br}

\DeclareMathOperator{\Cl}{Cl}
\DeclareMathOperator{\divv}{div}

\DeclareMathOperator{\Sym}{Sym}
\DeclareMathOperator{\Div}{Div}
\DeclareMathOperator{\Pic}{Pic}
\DeclareMathOperator{\Jac}{Jac}

\DeclareMathOperator{\Spec}{Spec}

\DeclareMathOperator{\red}{red}
\DeclareMathOperator{\res}{res}
\DeclareMathOperator{\cores}{cores}

\DeclareMathOperator{\rank}{rank}

\DeclareMathOperator{\GL}{GL}

\newcommand{\isom}{\cong}

\newcommand{\pcyclic}[2]{\langle #1, #2\rangle_p}
\newcommand{\mcyclic}[2]{\langle #1, #2\rangle_m}
\newcommand{\tcyclic}[2]{\langle #1, #2\rangle_2}
\newcommand{\scrA}{\mathscr{A}}
\newcommand{\scrC}{\mathscr{C}}

\numberwithin{equation}{section}
\numberwithin{table}{section}

\title{Unramified Brauer classes on cyclic covers of the projective plane}

\author{Colin Ingalls, Andrew Obus, Ekin Ozman, and Bianca Viray}

 \makeatletter
\let\@wraptoccontribs\wraptoccontribs
\makeatother

\contrib[with an appendix by]{Hugh Thomas}
\thanks{C.I. was partially supported by an NSERC Discovery Grant.  A.O. was partially supported by NSF Grant DMS-1265290. B.V. was  partially supported by NSF grant DMS-1002933.}

\subjclass[2010]{Primary: 14F22; Secondary: 12G05, 14J28, 14J50, 15A66, 16K50}
\keywords{Brauer group, Clifford algebra, K3 surface, $p$-cyclic cover, symbol algebra, quadratic form}

\begin{document}

	\begin{abstract}
	Let $X \to \PP^2$ be a $p$-cyclic cover branched over a smooth, connected curve $C$ of degree divisible by $p$, defined over a separably closed field of prime-to-$p$ characteristic.  We show that all (unramified) $p$-torsion Brauer classes on $X$ that are fixed by $\Aut(X/\PP^2)$ arise as pullbacks of certain Brauer classes on $\kk(\PP^2)$ that are unramified away from $C$ and a fixed line $L$.  We completely characterize these Brauer classes on $\kk(\PP^2)$ and relate the kernel of the pullback map to the Picard group of $X$.
	
	If $p = 2$, we give a second construction, which works over any base field of characteristic not 2, that uses Clifford algebras arising from symmetric resolutions of line bundles on $C$ to yield Azumaya representatives for the $2$-torsion Brauer classes on $X$.  We show that, when $p=2$, and $\sqrt{-1}$ is in our base field, both constructions give the same result.  
	\end{abstract}
	\maketitle
	
\section{Introduction}
	
	The Brauer group $\Br X$ of a smooth projective variety $X$ is a fundamental object of study and has a variety of applications.  For instance, Artin and Mumford used the birational invariance of the Brauer group to exhibit non-rational unirational three-folds~\cite{ArtinMumford}.  The Brauer group also encodes arithmetic information.  For a variety $X$ over a global field $k$, the Brauer group of $X$ can obstruct the existence of $k$-rational points, even when there are no local obstructions~\cite{Manin}.  In a different direction, Brauer groups are related to derived categories.  More precisely, derived categories of varieties are sometimes equivalent to derived categories of the category of sheaves of modules over an Azumaya algebra, or twisted sheaves; when this occurs these algebras and twistings are classified by the Brauer group.
	
	For each of these applications, it is often necessary to work with elements of the Brauer group quite explicitly, say as Azumaya algebras over $X$ or as central simple algebras over the function field of $X$ (We refer the reader to Section~\ref{sec:back} for background on the Brauer group and Brauer elements).  Unfortunately, there are few general methods for obtaining such representatives.  
	
	In 2005, van Geemen gave a construction of all $2$-torsion Brauer classes on a degree $2$ K3 surface of Picard rank $1$.  Precisely he proves:
	\begin{thm_nonum}[{\cite[\S9]{vanGeemen-Degree2}}]
		Let $\pi \colon X \to \PP^2_{\C}$ be a double cover branched over a smooth sextic curve $C$; then $X$ is a degree $2$ K3 surface.  Assume that $X$ has Picard rank $1$.  Then there is a natural isomorphism
		\begin{equation}\label{eq:IsomBr}
			\Br X[2] \stackrel{\sim}{\to} \left(\frac{\Pic C}{K_C}\right)[2].
		\end{equation}
		Moreover, for every $\alpha\in \Br X[2]$ there is a geometric construction of an Azumaya algebra on $X$ representing $\alpha$; the construction involves exactly one of
		\begin{enumerate}
			\item a double cover of $\PP^2\times\PP^2$ branched over a $(2,2)$ curve,
			\item a degree $8$ K3 surface, i.e., an intersection of $3$ quadrics in $\PP^5$, or
			\item a cubic $4$-fold containing a plane.
		\end{enumerate}
	\end{thm_nonum}
	
	The main result of this paper is an extension of van Geemen's result, \emph{including} the geometric constructions of Azumaya algebras, to degree $2$ K3 surfaces of any Picard rank and, more generally, to any smooth double cover of $\PP^2$. 
	\begin{thm}\label{thm:geometricBr}
		Let $k$ be a separably closed field of characteristic not $2$ and let $\pi \colon X \to \PP^2_{k}$ be a double cover branched over a smooth curve $C$ of degree $2d$.  Then there is an exact sequence
		\[
			0 \to \frac{\Pic X}{\Z H + 2\Pic X} \to 
			\left(\frac{\Pic C}{K_C}\right)[2] \stackrel{\scrA_X}{\longrightarrow} \Br X[2] \to 0.
		\]
		Moreover, for every $D\in\left(\Pic C/K_C\right)[2]$ there is a geometric description of an Azumaya algebra representing $\scrA_X(D)$; for $C$ in an open dense set of the moduli space of degree $2d$ plane curves, the construction involves exactly one of
		\begin{enumerate}
			\item a $(2,2)$ hypersurface of $\PP^{d - 1}\times \PP^2,$
			\item a $(2,1)$ hypersurface of $\PP^{2d - 1}\times \PP^2,$ or
			\item a cubic $(2d - 2)$-fold containing a $(2d - 4)$-dimensional linear space.
		\end{enumerate}
	\end{thm}
	Note that if $X$ is a degree $2$ K3 surface, then $d = 3$ and we exactly recover van Geemen's result. (In~Section\ref{sec:examples}, we show that the open dense set contains all degree $2$ K3 surfaces of Picard rank $1$.)
	\begin{remark}
		If one is interested only in extending~\eqref{eq:IsomBr}, then van Geemen states a similar result (he attributes the proof to the referee) that holds for more general surfaces, although still retaining an assumption on Picard rank~\cite[Thm. 6.2]{vanGeemen-Degree2}.  More generally, recent work of Creutz and the last author extend~\eqref{eq:IsomBr} and~\cite[Thm. 6.2]{vanGeemen-Degree2} to any smooth surface that is birational to a double cover of a ruled surface\cite[Thm. II]{CreutzViray}.  This result also gives a central simple $\kk(X)$-algebra representative of every $2$-torsion Brauer class; however, this $\kk(X)$-algebra may not extend to an Azumaya algebra over all of $X$.
	\end{remark}
	
	A result of Catanese (\cite{Cat}) on symmetric resolutions of line bundles allows us to extend van Geemen's geometric construction and obtain a \emph{set-theoretic} map
	\[
		\left(\frac{\Pic C}{K_C}\right)[2] \stackrel{\scrA_X}{\to} \Br X[2].
	\]
	Yet, it is difficult to discern from this construction whether $\scrA_X$ is even a group homomorphism, let alone prove surjectivity or determine the kernel.
	Van Geemen's proof of the isomorphism proceeds via the isomorphism between the Brauer group and the dual of the transcendental lattice of $X$.  The argument relies on a classification of the index $2$ sublattices of the transcendental lattice, which does not seem feasible for an arbitrary double plane $X$.  We avoid this classification by using ramified Brauer classes on $\PP^2.$  In doing so, we prove a result for $p$-cyclic covers of $\PP^2$.
	
	Let $X$ be a smooth projective geometrically integral $p$-cyclic cover of $\PP^2$ over a separably closed field $k$ of characteristic different from $p$, branched over a smooth curve $C$ of degree $pd$, for some $d$.  Let $L$ be a general line in $\PP^2$ that intersects $C$ in $pd$ distinct points, and let $U := \PP^2\setminus(C\cup L)$
	\begin{theorem}\label{thm:residueBr}
		Let $\zeta$ denote a primitive $p$th root of unity. There is an action of $\Z[\zeta]$ on $\Br X$ and we have a commutative diagram where the top row is exact.
		\[
				\xymatrix{
					0 \ar[r] & 
					\frac{\Pic X}{\Z H + (1 - \zeta)\Pic X} \ar[r] & 
					\left(\frac{\Pic C}{\Z L}\right)[p] 
					\ar[r]^{\psi} \ar@{^(->}^{\phi}[d] & 
					\Br X[1 - \zeta] \ar[r] \ar@{^(->}[d] & 0 \\
					& & \Br U[p]\ar[r]^{\pi^*}
					&\Br \kk(X)[p]
				}
		\] 
	\end{theorem}
	To complete the proof of Theorem~\ref{thm:geometricBr}, we show that the map $\scrA_X$ factors through a map $\scrA_U \colon \left(\frac{\Pic C}{K_C}\right)[2] \to \Br U[2]$, and then show that $\scrA_U$ and $\phi$ agree by comparing ramification data.
	
	We expect Theorem~\ref{thm:residueBr} to be of independent interest.  Even though we do not describe $\phi$ or $\psi$ by constructing central simple algebras or Azumaya algebras, the theorem should be helpful in constructing such representatives.  Indeed, this result shows that every $(1-\zeta)$-torsion Brauer class on $X$ can be obtained by pullback of a \emph{ramified Brauer class on $\PP^2$}, which are much simpler objects to understand.

	{\subsection{Related work and open questions}
	
		Some of the questions and ideas in this paper have been considered previously by a number of authors.  Although we do not need their work in the proofs, we would be remiss if we didn't mention them.  We do so briefly here.
	
		\subsubsection{Pulling back Brauer classes}
			The kernel of the map $\pi^*\colon (\Br \PP^2\setminus C)[2] \to \Br X$ was studied previously by Ford, who determined bounds for its dimension as an $\F_2$ vector space~\cite{Ford}.
	
		\subsubsection{Symmetric resolutions of line bundles}
		Our proof relies on work of Catanese about the existence of symmetric resolutions of line bundles.  These symmetric resolutions have been used by many authors,	sometimes to explicitly construct Brauer classes.  Examples include~\cites{Beauville-Determinantal, BCZ, ChanI}.}
	
	\subsection{Notation}
		Throughout, $p$ will be a fixed prime and $k$ will denote a separably closed field of characteristic $q \neq p$ ($q$ may be $0$).  
	
		Let $C = V(f)\subset \PP^2_k$ be a {smooth and irreducible} curve of degree $e$ and let $L = V(\ell)\subset \PP^2_k$ be a line that intersects $C$ in $e$ distinct points.  Define $U := \PP^2\setminus (C\cup L)$.  We will use $\tilde\ell$ to denote any linear form such that $V(\tilde\ell)\neq L.$
	
		Unless otherwise stated, we will assume that $e = pd$ for some integer $d$.  In this case, we can define $\pi\colon X\to \PP^2$ to be the $p$-cyclic cover of $\PP^2$ branched over $C$.  The morphism $\pi$ gives an isomorphism from $\tilde C := \left(\pi^{-1}(C)\right)_{\red}$ to $C$.
	
		We set $b_1(X)$ and $b_2(X)$ equal to the first and second Betti numbers of $X$, respectively.  That is, $b_i(X) = \dim \HH^i(X, \Q_r)$ for any prime $r \ne q$, when $i = 1, 2$.  These invariants are well defined since $X$ is a smooth surface (see, e.g., \cite[\S3.5]{Grothendieck2}).   Let the Hodge number $\hh^{0,1}(X) = \frac{1}{2}b_1(X)$ denote the irregularity of $X$, and the Hodge number $\hh^{2,0}(X)$ denote the geometric genus of $X$.  The Picard rank of $X$ is denoted $\rho(X)$.    

	For any field $F$, we let $F_s$ denote the separable closure. Note that all sheaves and cohomology are in the \'etale topology unless otherwise stated; we refer the reader to~\cite{Milne} for the background on \'etale cohomology. 
 
	\subsection{Outline}

		We begin by recalling some basic definitions and facts about Brauer groups and Clifford algebras in Section~\ref{sec:back}. In Section~\ref{sec:pcyclic}, we determine the numerical invariants of $X$ and show that there exists a natural ring action $\Z[\zeta]$ on $\Br X$, where $\zeta$ is a $p$th root of unity.  Section~\ref{sec:pullback} contains the proof of Theorem~\ref{thm:residueBr}.
		
		Starting with Section~\ref{sec:theta} we specialize to the case $p=2$; there we construct the maps $\scrA_X$ and $\scrA_U$. In Section~\ref{sec:main}, we prove that $\phi$ and $\scrA_U$ agree, thereby completing the proof of Theorem~\ref{thm:geometricBr}.  We specialize to degree $2$ K3 surfaces in Section~\ref{sec:examples} and show that the open dense subset in Theorem~\ref{thm:geometricBr} contains all Picard rank $1$ surfaces.  Lastly, we have an appendix by Hugh Thomas which proves a combinatorial proposition needed in Section~\ref{sec:theta}.

\section*{Acknowledgements}

	We thank Asher Auel and Brendan Hassett for expert advice.  This project began at the American Institute of Mathematics during a workshop on Brauer groups and obstruction problems.  We thank AIM for providing us with excellent working conditions.  We also thank the other members of our working group at the conference, Evangelia Gazaki, Alexei Skorobogatov, and Olivier Wittenberg, for helpful discussions.

\section{Background on the Brauer group {and Clifford algebras}} \label{sec:back}
Let $F$ be a field.
\begin{defn} Let $S$ be the monoid of isomorphism classes of central simple algebras over $F$ with finite rank, with operation given by tensor product. We say that two elements $A,B$ of $S$ are Morita equivalent if and only if the matrix algebras $M_{n}(A)$  and $M_k(B)$ are isomorphic for some integers $n,k$. Then the Brauer group $\Br F$ of $F$ is defined to be the quotient of $S$ by this equivalence relation. 
\end{defn}

Note that $\Br F$ is an abelian group under tensor product since tensor product commutes with taking matrix algebras. The inverse of an element $[A] \in \Br F$ is $[A^{op}]$ where $A^{op}$ is the opposite algebra obtained by reversing the order of multiplication in $A$.

The following is a classical fact from class field theory that relates Brauer groups to cohomology. 

\begin{theorem}[{\cite[Corollary 3.16]{MilneCFT}}]
For any separable closure $F_s$ of $F$ there exists a natural isomorphism between $\Br F$ and $\HH^2(\Gal(F_s/F), F^\times_s)$.
\end{theorem}

The definition of the Brauer group can be extended to a scheme $Y$ as in the series~\cites{Grothendieck1,Grothendieck2,Grothendieck}.

An $\calO_Y$-algebra $\calA$ is an Azumaya algebra over $Y$ if it is a coherent {locally free} $\calO_Y$-module and the specialization of $\calA$ at each point $y \in Y$ is a central simple algebra over the residue field at $y$. Two Azumaya algebras $\calA$, $\calA'$ over $X$ are Morita equivalent  if there exist locally free $\calO_Y$-modules $\calE$, $\calE'$ of finite rank such that $\calA \otimes_{\calO_Y} \End(\calE) $  is isomorphic to $\calA' \otimes_{\calO_Y} \End(\calE') $. This relation induces an equivalence relation and the tensor product operation carries over the quotient as in the case of central simple algebras over fields. The Azumaya Brauer group $\Br_{Az} Y$ of $Y$ is the abelian group given by the equivalence classes of Azumaya algebras.

The cohomological Brauer group of $Y$, denoted by $\Br_{\textup{cohom}} Y$, is the torsion part of the second cohomology group with values in $\G_m$, the sheaf of units. If $Y$ is smooth, this cohomology group is torsion, so $\Br_{\textup{cohom}} Y = \HH^2(Y, \G_m)$.  Furthermore, if $Y$ is a normal surface, if $Y$ is the separated union of two affine schemes, or if $Y$ is quasi-projective, then by Gabber's theorem the cohomological Brauer group is the same as the Azumaya Brauer group~\cite{deJong}.

Henceforth, we will only consider smooth quasiprojective varieties in which case all the above notions agree, thus we will simply refer to the Brauer group of $Y$ or $\Br Y$.

\begin{remark}
We have $\Br (\Spec F) = \Br F$ since $\HH^2(\Spec F, \G_m) \cong \HH^2(\Gal(F_s/F),F^\times_s)$. 
\end{remark}

It is possible to relate $\Br Y$ and the Brauer group of its function field $\Br \kk(Y)$ using the functoriality of \'etale cohomology. If $Y$ is an integral scheme, then using the inclusion $\Spec \kk(Y) \rightarrow X$, we obtain a map $\Br Y \rightarrow \Br \kk(Y)$. 

\begin{theorem}[{\cite[Chap. IV, Cor. 2.6]{Milne}}]\label{thm:inj}
Let $Y$ be a regular integral scheme with function field $\kk(Y)$.  Then the induced map $\Br Y \rightarrow \Br \kk(Y)$ is an injection.
\end{theorem}

Using Theorem \ref{thm:inj}, we will regard elements of $\Br Y$ as elements of the Brauer group of a field, $\Br \kk(Y)$. 

\subsection{Cyclic algebras}

It is a deep theorem of Merkurjev-Suslin that the Brauer group is generated by a special type of central simple algebra, namely the cyclic algebras.  In addition, cyclic algebras are the central simple algebras that are easiest to construct and study.  We recall their basic properties.

\begin{defn}
Let $F'/F$ be a cyclic Galois extension of degree $m$, fix an isomorphism $\chi \colon \Gal(F'/F) \rightarrow \Z/m\Z$, {and let $\tau := \chi^{-1}(1)$.} Let $b\in F^\times$.  Then the cyclic algebra associated to $\chi$ and $b$ is
\[
	\mcyclic{\chi}{b} := \frac{F[x;\tau]}{(x^m - b)},
\]
where $F[x;\tau]$ denotes the twisted polynomial ring, i.e., $\alpha x=x\tau(\alpha)$ for any $\alpha\in F.$
\end{defn}

\begin{remark}
	By abuse of notation, we sometimes use $\mcyclic{F'/F}{b}$ to denote $\mcyclic{\chi}{b}$, even though the cyclic algebra does depend on the choice of $\chi$ (equivalently $\tau$).
\end{remark} 

If $k'/k$ is a cyclic extension of degree $m$, we will sometimes denote the cyclic algebra $\mcyclic {\kk(Y_{k'})/\kk(Y)}{f}$ as $\mcyclic {k'/k}{f}$ since $\Gal(\kk(Y_{k'})/\kk(Y))$ is isomorphic to $\Gal(k'/k)$.

\begin{example}
 Every quaternion algebra is cyclic if the characteristic of the field $k$ is not $2$.
\end{example}

The following presentation of a cyclic algebra is useful for computations. We refer to \cite[Sections 2.5, 4.7]{CSAandGalC} for its proof.

\begin{proposition}\label{coh} 
Assume that the characteristic of $F$ is relatively prime to $m$ and fix $b \in F^\times$, $F'/F$ as above.  We consider $b$ as representing an element of $\HH^1(F,\mu_m)$ via the Kummer isomorphism $F^\times/F^{\times m} \simeq \HH^1(F,\mu_m)$.  Let $\chi$ be an element in $ \Hom(\Gal(F_s/F) , \Z/m\Z)$ which is isomorphic to $\HH^1(F,\Z/m\Z)$.  Consider the cup product $ \HH^1(F,\Z/m\Z)  \times \HH^1(F,\mu_m) \rightarrow \HH^2(k,\mu_m) \simeq \Br F[m]$. The image of the tuple $(\chi,b)$ under this cup product is the cyclic algebra $\mcyclic{\chi}{b}$. 
\end{proposition}

If $F$ contains a primitive $m^{th}$ root of unity (in particular, this means $\Char(F) \nmid m$), then we may define $m$-symbol algebras.
\begin{defn}
	Assume that $F$ contains a primitive $m^{th}$ root of unity and let $a,b,\in F^{\times}$.  Then the $m$-symbol algebra $\mcyclic{a}{b}$ is the $F$-algebra generated by $x$ and $y$ such that $x^m=a, y^m=b$ and $xy=\zeta_m yx$ where $\zeta_m$ is a primitive $m$th root of unity.
\end{defn}
\begin{theorem}\label{symbol}
Assume that $F$ contains a primitive $m^{th}$ root of unity and let $a,b,\in F^{\times}$. Then any cyclic algebra $\mcyclic {\chi}{b}$ is isomorphic to the symbol algebra $\mcyclic {a}{b}$ for some $a \in F^\times$.
\end{theorem}

\begin{proof}
It is easier to see this using the cohomological interpretation given in Proposition \ref{coh}. Since $F$ contains a root of unity $\zeta_m,$ we have an isomorphism $\Z/m\Z  \rightarrow \mu_m$ as $\Gal(F_s/F)$-modules. This induces the following isomorphisms: $\HH^1(F,\Z/m\Z) \simeq \HH^1(F,\mu_m) \simeq F^\times/(F^\times)^m$. Say $a,b \in F^\times$ are fixed. The elements $a,b$ can be seen in the corresponding cohomology groups using the mentioned isomorphisms. Then using the cup product $\HH^1(F,\Z/m\Z)  \times \HH^1(F,\mu_m) \rightarrow \HH^2(F,\mu_m)$ we get an element of  $\HH^2(F,\mu_m) \simeq \Br F [m]$.
\end{proof}

Now we will list some basic properties of cyclic and symbol algebras which will be useful in later sections. For proofs of these facts we refer to \cite{CSAandGalC}, \cite{Saltman}, and \cite{Serre}.

\begin{defn} A central simple algebra $A$ over $F$ is called split if $A$ is isomorphic to the matrix algebra $M_n(F)$.  If $F'$ is a field such that the $F'$-algebra $A \otimes_F F'$ is isomorphic to $M_r(F')$ for some integer $r \geq 1$, then we say that $F'$ splits $A$.
\end{defn}

\begin{proposition}\label{prop:BrauerFacts}
 Let $A$ be central simple algebra over a field $F$.
 \begin{enumerate}

\item[i)]  $A$ is a cyclic algebra if and only if there exists a cyclic extension of $F$ splitting $A$. More precisely,
$A$ is split by $F'$, an $m$-cyclic Galois extension of $F,$ if and only if $A$ is isomorphic to $\mcyclic{\chi}{b}$ for some $b \in F^\times$ and some isomorphism $\chi \colon \Gal(F'/F) \to \Z/m$. 

\item[ii)] If $A$ is a cyclic algebra over $F$ and $F'$ is a field extension of $F,$ then $A \otimes_F F'$ is a cyclic algebra over $F'$.

\item[iii)] In the Brauer group of $F$, 
\[
	[\mcyclic{\chi}{a}]+[\mcyclic{\chi}{b}]=[\mcyclic{\chi}{ab}] 
	\quad \textup{ and } \quad
	[\mcyclic{\chi}{a}]+[\mcyclic{\psi}{a}]=[\mcyclic{\chi+\psi}{a}] 
\]
for all $a,b \in F^\times$ and $\chi,\psi \in \HH^1(G,\Z/m\Z)$, where $G=\Gal(F'/F)$. In particular, if $F$ contains a primitive $m^{th}$ root of unity, then for all $a,b,c\in F^\times$
\[
	[\mcyclic{a}{b}]+[\mcyclic{a}{c}]=[\mcyclic{a}{bc}] 
	\quad \textup{ and } \quad
	[\mcyclic{a}{c}]+[\mcyclic{b}{c}]=[\mcyclic{ab}{c}]. 
\]
 
\item[iv)] \label{norm} Two cyclic algebras $(\chi,a)$ and $(\chi,b)$ are equivalent if and only if $a/b $ is the norm of an element in $F'$, where $F'$ is an $m$-cyclic Galois extension of $F$ and $\chi \colon \Gal(F'/F) \rightarrow \Z/m\Z$ is a fixed isomorphism. 

\item[v)] \label{trivial} If $F$ contains a primitive $m^{th}$ root of unity, then $\mcyclic{a}{(-a)^n}$ is trivial in $\Br F$ for all $a\in F^{\times}$ and $n\in \Z$.

 \end{enumerate}
 
\end{proposition}

For completeness we restate the theorem of Merkurjev and Suslin which was mentioned in the beginning of the section.

\begin{theorem}[{Merkurjev-Suslin, see e.g.~\cite[Thm. 4.6.6]{CSAandGalC}}]\label{MS}
Assume that $\Char(F)\nmid m$. Then any representative of a class of order dividing $m$ in $\Br F$ is Morita equivalent (but not necessarily isomorphic)  to the tensor product of cyclic algebras. 
\end{theorem}

\subsection{Residues of Brauer elements} \label{sec:res}

An important map that we need in the next sections is the residue (or ramification) map. We will give the explicit description of this map for cyclic algebras. Using Theorem \ref{MS} it is possible to extend this map to any element of $\Br \kk(Y)$ whose order is not divisible by the characteristic. 

 We begin with defining this map for a complete discrete valuation ring $R$ with field of fractions $K$. Let $v$ denote the discrete valuation associated to $R$ and $\pi$ be a uniformizer such that $v_R(\pi)=1$. The residue field $R/\pi R$ is denoted by $F$ and has characteristic $q$.

\begin{theorem}\label{thm:unr}
Let $K_{un}$ denote the maximum unramified extension of $K$. Let $[A] \in \Br K$ be an element of prime order $q$. Then $[A]$ is split by $K_{un}$. 
\end{theorem}

\begin{proof}
This is Theorem 10.1 in \cite{Saltman}.
\end{proof}

For any abelian torsion group $G$, let $G'$ denote the prime-to-$q$ component.
Since $K$ is complete, the valuation $v$ extends uniquely to $K_{un}$ and gives a $\Gal(K_{un}/K)$-equivariant map between $K_{un}$ and $\Z$ where the action on $\Z$ is the trivial action.  Therefore, there exists an induced map $f_1 \colon \HH^2(\Gal(K_{un}/K),K_{un}^\times)' \rightarrow \HH^2(\Gal(K_{un}/K),\Z)' $.

Now consider the short exact sequence $$0 \rightarrow \Z \rightarrow \Q \rightarrow \Q/\Z \rightarrow 0.$$ The associated long exact sequence gives a coboundary map, which is an isomorphism $$\HH^1(\Gal(K_{un}/K), \Q/\Z) \xrightarrow {f_2}\HH^2(\Gal(K_{un}/K),\Z).$$

Using these arguments, we can define the residue map as follows:

\begin{defn}
The residue (ramification) map $\partial_v \colon (\Br K)' \rightarrow \Hom(\Gal(K_{un}/K),\Q/\Z)'$ is defined as the composition:

 \begin{eqnarray*} 
 \Br(K)' \cong \HH^2(\Gal(K_{un}/K),K_{un}^\times)' & \xrightarrow{f_1} & \HH^2(\Gal(K_{un}/K),\Z)'  \\
 &\cong & \HH^1(\Gal(K_{un}/K), \Q/\Z)' = \Hom(\Gal(K_{un}/K), \Q/\Z)' 
 \end{eqnarray*}
\end{defn}

The first isomorphism follows from Theorem \ref{thm:unr}, and the second isomorphism is given by $f_2^{-1}$ as described above. The last equality follows from the definition of group cohomology. Note that $\Hom$ denotes continuous homomorphisms.

As seen from the definition, computing the residue map is not easy for a general element of the Brauer group. However, for cyclic algebras it  can be done as follows. Let $R,K,v,F,\pi$ be as above and let $L$ be a cyclic Galois extension of $K$ of degree $m$ where $m$ is relatively prime to $q$. Assume that the residue field extension is also cyclic of degree $m$.  Fix an isomorphism $\chi \colon \Gal(L/K) \rightarrow \Z/ m\Z$ and let $\tau=\chi^{-1}(1)$.

\begin{proposition}\label{resfield}
The map $\partial_v(\mcyclic{\chi}{\pi})\in \Hom(\Gal(K_{un}/K), \Q/\Z)$ factors through $\Gal(L/K)$ and it sends $\tau\in \Gal(L/K)$ to $1/m +\Z.$
\end{proposition}

\begin{proof}
We will sketch the proof, for more details see \cite{Saltman}.

Note that by assumption $L/K$ is unramified of degree $m$. Let $G$ denote the Galois group $\Gal(L/K)$. The image  of $[A]$ under the residue map is the image of $[A]$ under the isomorphisms $\HH^2(G,L^\times) \rightarrow \HH^2(G,\Z) \cong \HH^1(G,\Q/Z)$ using the naturality of the inflation map. Note that $\HH^2(G,L^\times)=(L\times)^G/N_G(L^\times)=K^\times/N_G(L^\times)$ and $\HH^2(G,\Z)=\Z^G/N_G(\Z)=\Z/n\Z$, where $N_G$ denotes the norm map. Then $[A] \in \Br(K)$ corresponds to $\pi \in K^\times$, which maps to $1 \in \Z$. The rest follows from the definition of the coboundary map $f_2 \colon \HH^1(\Gal(K_{un}/K), \Q/\Z) \rightarrow \HH^2(\Gal(K_{un}/K),\Z)$.
\end{proof}

In fact the residue map is onto and its kernel is the prime-to-$q$ part of the Brauer group of $R$.

\begin{theorem}[Theorem 10.3, \cite{Saltman}]
The following sequence is exact:
$$0 \rightarrow (\Br R)' \rightarrow (\Br K)' \rightarrow \Hom(\Gal(K_{un}/K), \Q/\Z)' \rightarrow  0$$
\end{theorem}

Now we will define the residue map for cyclic Brauer classes over the function field of a surface.  By the Merkurjev-Suslin theorem (Theorem~\ref{MS}), this defines it on all Brauer classes of order {prime to the characteristic}. Let $Y$ be a complete, smooth algebraic surface over an separably closed field $k$. Let $Y_1$ denote the set of codimension one points, i.e.~the set of all irreducible curves on $Y$. Then the following sequence is exact by \cite{Grothendieck1, Grothendieck2, Grothendieck}:  
$$0 \rightarrow \Br Y \rightarrow \Br \kk(Y) \xrightarrow{\oplus_C\partial_C} \bigoplus _{ C \in Y_1} \HH^1(\kk(C), \Q/\Z). $$ 

The map $\Br Y \rightarrow \Br \kk(Y)$ is the injection map of Theorem \ref{thm:inj}. The sum is taken over all irreducible curves on $Y$, and $\kk(C)$ denotes the field of rational functions on $C$. Let $\chi \in \HH^1(\kk(C),\Q/\Z)=\Hom(\Gal(\kk(C)_s/\kk(C)),\Q/\Z)$. Then $\chi$ has a kernel $\Gal(L/\kk(C))$ where $L/\kk(C)$  is cyclic Galois of degree $m$. Therefore $\chi$ is determined by choosing a generator $\tau$ of $\Gal(L/\kk(C))$ such that $\chi(\tau)=\frac{1}{m}+ \Z$, similarly to the case in Proposition \ref{resfield}.

By Theorem \ref{symbol}, every cyclic algebra of order $m$ is a symbol algebra.
The image of $\mcyclic{a}{b}$ under $\bigoplus_C\partial_C$ is the tuple of residues of $\mcyclic{a}{b}$ along $C$ for every irreducible curve $C\subset X$.

\begin{theorem} \label{thm:res}
 The residue of  $\mcyclic{a}{b}$ along $C$ is the cyclic Galois extension $L$ of $\kk(C)$ obtained by adjoining the $m$th root of $a^{v_C(b)}b^{-v_C(a)}$. 
\end{theorem}

A rigorous proof of this theorem relies on Milnor $K$-theory and can be found on Section 7.5.1 of \cite{CSAandGalC}. 
For a similar computation see \cite[p.~10]{CTOP}.

\begin{defn}The combination of curves $C$ in $Y_1$ such that $\partial_C(A) \neq 0$ is called the ramification divisor of $A$ on $Y$.
\end{defn}

\begin{remark}
If $\mcyclic{a}{b}$ has nontrivial ramification along $C$, then $C$ is  a prime divisor of $a$ or $b$. 
\end{remark}

\subsection{Symbols and residues of Clifford algebras}\label{sec:clifford}

	In this section, we will recall and prove some preliminary results about symbols and residues of Clifford algebras, which will be helpful for studying $2$-torsion Brauer classes.  Throughout, we work over a field $F$ of characteristic different from $2$.  For simplicity, we assume that $\sqrt{-1}\in F$; a similar analysis can be carried out in the case that $\sqrt{-1} \notin F$, but is more complicated to state.  By an abuse of notation, we will conflate diagonal quadratic forms in $n$ variables and diagonal $n\times n$ matrices; they will both be denoted by
	\[
		[\alpha_1, \alpha_2, \ldots, \alpha_n],\quad\textup{with }\alpha_i\in F.
	\]

	First we recall a well-known proposition which follows from Gram-Schmidt.
	\begin{proposition}\label{prop:Diagonal}
			Let $M=(m_{ij})$ be a symmetric $n\times n$ matrix over a field $F$ with
		corresponding quadratic form $Q$ and let $M_{i}$ denote the upper left 
		$i\times i$ principal minor.
		If $M_i\neq 0$ for all $i$, then $Q$ is equivalent to the diagonal quadratic form 
		$$ [ M_{1},M_{2}/M_{1},\ldots,M_n/M_{n-1} ].$$
	\end{proposition}

	Given a non-degenerate quadratic form $Q$, we may consider the Clifford algebra $\Cl(Q)$, a finite dimensional simple algebra over $F$. The multiplication in $\Cl(Q)$ encodes the geometry determined by $Q$ on $V$.
	More precisely, let $Q\colon V \rightarrow F$ be a non-degenerate quadratic form on a finite dimensional $F$-vector space $V$.

	\begin{defn}
Let $(V,Q)$ be a quadratic space and $\mathcal{T}(V):=\bigoplus_{p=0}^{\infty} \bigotimes^p V$ be its tensor algebra. The Clifford algebra $\Cl(Q)$ is $\mathcal{T}/I_Q$, where $I_Q$ is two sided ideal in $\mathcal{T}(V)$ generated by $\{v \otimes v -Q(v) | v \in V\}$.
	\end{defn}
	
	Clifford algebras can be characterized as follows: 
	
	\begin{prop} \label{prop:cliff}Let $(V,Q)$ be a quadratic space. Let $\scrC$ be an $F$-algebra with unit $1_\scrC$. Then $\scrC$ is a Clifford algebra for $Q$ if it has the following properties:
	\begin{enumerate}
 \item There exists an  $F$-module map $f_\scrC\colon V \rightarrow \scrC$ such that $f_\scrC(v)^2=Q(v)1_\scrC$ for any $v $ in $V$.\label{comp}
	\item $\scrC$ is `minimal' with respect to the above property, i.e., for any $F$-algebra $A$ that satisfies Property (\ref{comp}), there exists a unique $F$-algebra homomorphism $\phi\colon \scrC \rightarrow A$, such that $f_A(v)=\phi(f_\scrC(v))$ for any $v \in V$. 
	\end{enumerate}
	\end{prop}
	\begin{proof} 
		See pages 103-104 of \cite{Lam}.
	\end{proof}

	As a corollary we see that equivalent quadratic forms give isomorphic Clifford algebras.
	\begin{cor}\label{cor:cliff}
		Let $\iota$ be an isometry of $(V,Q)$ into $(W,Q')$. Then there exists an algebra isomorphism $\iota_{\Cl}: \Cl(Q) \rightarrow \Cl(Q')$ extending $\iota$ such that the following diagram commutes:
		\[
			\xymatrix{
				{\Cl(Q)} \ar@{->}[r]^{\iota_{\Cl}} & \Cl(Q')\\
	   		 	{V}\ar@{->}[r]^{\iota}\ar@{->}[u]& {W}\ar@{->}[u]
			 }
		\]
	\end{cor}
	\begin{proof} Follows from Proposition \ref{prop:cliff}.
	\end{proof}

	\begin{remark}
		Let $B_Q$ be the associated symmetric bilinear form defined by $B_Q(v_1,v_2)=Q(v_1+v_2)-Q(v_1)-Q(v_2)$. By Property~(\ref{comp}) of Proposition~\ref{prop:cliff}, $v_1v_2=-v_2v_1$ if and only if $v_1$ and $v_2$ are orthogonal with respect to $B_Q$. Therefore, we can say that the multiplication on the Clifford algebra encodes the geometry determined by $Q$ on $V$.
	\end{remark}
	The tensor algebra $\mathcal{T}(V)$ has a $\Z/2\Z$-grading given by $\mathcal{T}_0(V)=\bigoplus_{i \geq 0, \textbf{even}} \bigotimes^i V$ and  $\mathcal{T}_1(V)=\bigoplus_{i \geq 1, \textbf{odd}} \bigotimes^i V$. There is a canonical map from $\mathcal{T}(V)$ onto $\Cl(Q)$. Using this map and the grading on $\mathcal{T}$, one can give a $\Z/2\Z$-grading on $\Cl(Q)$ such that $\Cl(Q) = \Cl_0(Q) \oplus \Cl_1(Q)$. The part $\Cl_0(Q)$ is a subalgebra; we refer to is as the \textit{even Clifford algebra} of $Q$.

	\begin{remark}\label{rem:cliff}
	Another way to see the grading on the Clifford algebra is using the isometry $v \mapsto -v$. This isometry extends to an involution of $\Cl(Q)$ using Corollary~\ref{cor:cliff}. 
	\end{remark}

	The Clifford algebra and even Clifford algebra of $Q$ have the following properties.
	\begin{proposition}\label{Pcliffassymbol}
		Let $Q$ be the diagonal quadratic form $[m_1, \ldots, m_n]$.
		\begin{enumerate}
			\item $\Cl_0(Q) = \Cl([
		-m_1m_2,\ldots,-m_1m_n ]).$
			\item If $n$ is odd, then $\Cl_0(Q)$ is a central simple algebra over $F$ and in $\Br F$ we have
			\[
				\Cl_0(Q) = \sum_{1 \leq i < j \leq n}\tcyclic{m_i}{m_j}.
			\]
			\item If $n$ is even, then $\Cl(Q)$ is a central simple algebra over $F$ and in $\Br F$ we have
			\[
				\Cl(Q) = \sum_{1 \leq i < j \leq n}\tcyclic{m_i}{m_j}.
			\]
		\end{enumerate}
	\end{proposition}
	\begin{proof}
		\cite[V.2.4, V.2.10, V.3.20]{Lam}
	\end{proof}
	
	It follows from part 3 of above proposition and Remark~\ref{rem:cliff} that if the dimension of $(V,Q)$ is even then $\Cl(Q)$ determines a two torsion element in $\Br F$. Moreover, by the following celebrated result of Merkurjev we know that all two torsion is determined this way.
	
	\begin{theorem}[Merkurjev, \cite{M1}] Let $F$ be a field with characteristic different than $2$. Then the two torsion of the Brauer group of $F$ is generated by Clifford algebra of some even dimensional quadratic form.
	\end{theorem}
	
	Finally, the proposition below, which follows from Proposition~\ref{Pcliffassymbol} via a straightforward computation, determines the residue map $\partial_v$ on Clifford algebras.
	
	\begin{prop}\label{Presvals}
		Let $F$ be a field with a discrete valuation $v$.  Let $M=(m_{ij})$ be a symmetric $n\times n$ matrix over a field $F$ with corresponding quadratic form $Q$ and let $M_{i}$ denote the upper left $i\times i$ principal minor.  Assume that $M_i\neq 0$ for all $i$, and let $e_i := v(M_i)$.  We define $e_0 = e_{n+1} = 0$.  Then
		\begin{enumerate}
			\item If $n$ is even, then $\partial_v(\Cl(Q)) = \prod_{i=1}^n M_i^{e_{i+1}-e_{i-1}}$.
			\item if $n$ is odd, then $\partial_v(\Cl_0(Q)) = \prod_{i=1}^n M_i^{e_{i+1}-e_{i-1}}$.
		\end{enumerate}
	\end{prop}
	\begin{proof}
		By Proposition~\ref{prop:Diagonal}, we have
 $$\Cl(Q) = \Cl([M_1, M_2/M_1, \ldots, M_n/M_{n-1}])$$ $$\Cl_0(Q) = \Cl_0([M_1, M_2/M_1, \ldots, M_n/M_{n-1}]).$$  Then, by Proposition~\ref{Pcliffassymbol}, it remains to compute $\partial_v(A)$ where $A := \sum_{1 \leq i < j \leq n}\tcyclic{M_i/M_{i-1}}{M_j/M_{j-1}}$ and $M_0 := 1$.  By Theorem~\ref{thm:res},
		\[
			\partial_v(A) = \prod_{1 \leq i < j \leq n} 
			\left(\frac{M_i}{M_{i-1}}\right)^{e_j - e_{j-1}}
			\left(\frac{M_j}{M_{j-1}}\right)^{e_{i-1} - e_i}.
		\]
		We may rewrite this as
		\begin{align*}
			\partial_v(A) = & M_1^{e_2 - e_0}M_n^{e_{n+1} - e_{n-1}}\prod_{k = 2}^{n - 1}M_k^{\sum_{j = k+1}^n (e_j - e_{j-1}) + \sum_{j = k + 2}^n (e_{j-1} - e_j) + \sum_{i = 1}^{k-1} e_{i-1} - e_i + \sum_{i = 1}^k e_i - e_{i-1}}\\
			= & M_1^{e_2 - e_0}M_n^{e_{n+1} - e_{n-1}}\prod_{k = 2}^{n - 1}M_k^{(e_n - e_{k}) + (e_{k+1} - e_n) + (e_0 - e_{k-1}) + (e_k - e_0)} = \prod_{i = 1}^nM_i^{e_{i+1} - e_{i-1}},
		\end{align*}
		which completes the proof.
	\end{proof}

\section{Preliminaries on $p$-cyclic covers} \label{sec:pcyclic}

	Recall that $X$ is a $p$-cyclic cover of $\PP^2$ branched at $C$.
	In this section, we establish some numerical invariants of $X$ that will be of use later.  We also show that there is a natural ring action
	of $\Z[\zeta]$ on $\Br X$, where $\zeta$ is a primitive $p$th root of unity.  

    \subsection{Geometric invariants of $X$}
	\begin{prop}\label{prop:NumericalInvariants}
		We have that $\hh^{0,1}(X) = 0$, that $b_2(X) = 1 + (p - 1)\left(1 + 2g(C)\right), $ and that 
				\[
					\hh^{2,0}(X) = p - 1 - \frac32d\sum_{j = 1}^{p-1}j + 
					\frac12d^2 \sum_{j = 1}^{p-1}j^2,
				\]
				where $g(C)$ denotes the genus of $C$.
	\end{prop}
	\begin{proof}
		Since $\pi \colon X \to \PP^2$ is an affine morphism, by~\cite[Chap. III, Ex. 8.2]{Hartshorne}, we have that
		\[
			\HH^i(X,\OO_X) \isom \HH^i(S, \pi_*\OO_X),\quad i = 0,1,2.
		\]
		Furthermore, by~\cite[Chap. I, Lemma 17.2]{BHPVdV}, $\pi_*\OO_X \isom \bigoplus_{j = 0}^{p - 1}\calL^{-j}$, where $\calL = \OO(d)$.	
Thus, we have
		\begin{align*}
			\chi(X) & = \hh^0(X,\OO_X) - \hh^1(X, \OO_X) + \hh^2(X, \OO_X)\\
			& = \hh^0\left(S,\oplus_{j = 0}^{p - 1}\calL^{-j}\right) 
			- \hh^1\left(S, \oplus_{j = 0}^{p - 1}\calL^{-j}\right) 
			+ \hh^2\left(S, \oplus_{j = 0}^{p - 1}\calL^{-j}\right)
			 = \sum_{j = 0}^{p-1}\chi(S, \calL^{-j}).
		\end{align*}
		A similar argument shows that 
		\[
			\hh^{2,0}(X)  = \hh^2(S, \OO_S) + \hh^2(S, \mathcal{L}^{-1}) + \cdots +
			\hh^2(S, \mathcal{L}^{1-p}).
		\]
	
		Since $\hh^1(\PP^2, \OO(n)) = 0$ for all $n$ and $\hh^0(\PP^2, \OO(n)) = 0$ for all $n<0$, 
		$$\hh^{0,1}(X) = 1 + \hh^{2,0}(X) - \chi(X) = 1 - \hh^0(\PP^2,\OO_{\PP^2}) = 0.$$  
		The formula for $\hh^{2,0}(X)$ is deduced from Serre duality together with the formula $\hh^0(\PP^2, \OO(n)) = \binom{n+2}{2}$ for $n \geq -2$.
		
		Lastly, we prove the formula for $b_2(X)$.  By Noether's formula,
		\[
			b_2(X) := 12\chi(X) - 2 - K_{X}^2
		\]
		Furthermore, $K_X = \pi^*(K_{\PP^2}\otimes\calL^{p - 1}) = \pi^*\OO(pd - d - 3)$~\cite[Chap. I, Lemma 17.1]{BHPVdV}.  Therefore, by the projection formula,
		\[
			K_X^2 = p\left(pd - d - 3\right)^2 = p(p-1)^2d^2 - 6p(p-1)d + 9p.
		\]
		So
		\begin{align*}
			b_2(X) & = 12\chi(X) - 2 - K_X^2 = 10 + 12h^{2,0}(X) - K_X^2\\
			& = 10 + 12p - 12 - 18d\sum_{j=1}^{p-1}j + 6d^2\sum_{j=1}^{p-1}j^2 -p(p-1)^2d^2 + 6p(p-1)d - 9p\\
			& = 3p - 2 + d\left(6p(p-1)d - 18\sum_{j=1}^{p-1}j\right)
			+ d^2\left(-p(p-1)^2+ 6\sum_{j=1}^{p-1}j^2\right)\\
			& = 3p - 2 - 3p(p-1)d + p^2(p-1)d^2 = 1 + (p-1)\left(1 + (pd - 1)(pd - 2)\right).
		\end{align*}
		We complete the proof by noting that $g(C) = \binom{pd-1}{2}.$
	\end{proof}

	\begin{prop}\label{Ptriv}
	The fundamental group $\pi_1(X)$ is trivial.
	\end{prop}
	\begin{proof}
	Since $X \to \P^2$ is tamely ramified over a smooth, irreducible curve, this follows from \cite[Proposition 8 (iii)]{Abhyankar}.
	\end{proof}
	\begin{cor}\label{CpicH3}
		The groups $\Pic X [p^n]$, $\HH^1(X, \Z/p^n)$, and $\HH^3(X, \Z/p^n)$ are trivial for all $n$.
	\end{cor}
	\begin{proof}
		Nontrivial elements of $\Pic X[p^n]$ give rise to nontrivial cyclic \'{e}tale covers of $X$ of order $p^n$.   
		Proposition \ref{Ptriv} shows that no such covers exist. Also, $\HH^1(X, \Z/p^n) = \Hom(\pi_1(X), \Z/p^n) = 0$ for
		all $n$.  Since $k$ contains all of the $p^{\infty}$th roots of unity, Poincar\'{e} duality for \'{e}tale cohomology shows that
		$\HH^3(X, \Z/p^n) = 0$ for all $n$.  
		\end{proof}
	
	\begin{lemma}\label{Lbetti}
	We have $\HH^2(X, \Z/p) \cong (\Z/p)^{b_2(X)}$.
	\end{lemma}
	\begin{proof}
	Let $b$ be such that $\HH^2(X, \Z/p) \cong (\Z/p)^b$.  It suffices to show that $\HH^2(X, \Z/p^n) \cong (\Z/p^n)^b$ for all $n$, as then
	$b$ will be the Betti number we seek.  We proceed by induction.  Consider the exact sequence
	$$0 \to \Z/p^n \to \Z/p^{n+1} \to \Z/p \to 0.$$
	Using the long exact sequence for cohomology, as well as Corollary \ref{CpicH3}, we get a short exact sequence
	$$0 \to \HH^2(X, \Z/p^n) \to \HH^2(X, \Z/p^{n+1}) \to \HH^2(X, \Z/p) \to 0,$$ which by induction can be written as
	$$0 \to (\Z/p^n)^b \to \HH^2(X, \Z/p^{n+1}) \to (\Z/p)^b \to 0.$$  Thus
	$$\HH^2(X, \Z/p^{n+1}) \cong (\Z/p^n \oplus \Z/p)^s \oplus (\Z/p^{n+1})^t,$$ where $s + t = b$.  The map to $\HH^2(X, \Z/p)$ is reduction 
	modulo $p$, which yields an image of $(\Z/p)^{2s+t}$.  
	So $s + t = 2s + t = b$, and we conclude that $s = 0$ and $t = b$, completing the induction.	
	\end{proof}
	
	\begin{cor}\label{cor:Brdim}
		We have $\dim_{\F_p}\Br X[p] = (p - 1)(1 + 2g(C)) + 1 - \rho(X)$, where $\rho(X) = \textup{rank }\textup{NS}(X)$.
	\end{cor}
	\begin{proof}
		
		By Proposition~\ref{prop:NumericalInvariants}, we need only prove that $\dim_{\F_p}\Br X[p] = b_2(X) - \rho(X)$.  Using the Kummer 
		exact sequence on $X$:
		$$1 \to \Z/p \to \G_m \to \G_m \to 1,$$ we obtain the following from the long exact sequence:
		$$0 \to \Pic X/p \cdot \Pic X \to \HH^2(X, \Z/p) \to \Br X[p] \to 0.$$
		By Corollary~\ref{CpicH3}, the first term is $(\Z/p)^{\rho(X)}$.  By Lemma~\ref{Lbetti}, the middle term is $(\Z/p)^{b_2(X)}$.  The corollary follows.
	\end{proof}
	
	\subsection{The cyclic action on the Brauer group}
		Let $\zeta$ denote a generator of the Galois group of $\pi \colon X \to \PP^2$.  Then $\zeta$ acts on $\Br X$ as follows.  Embed $\Br X$ into $\Br \kk(X) = \HH^2(\kk(X), (\kk(X)_s)^{\times})$.  Extend the canonical automorphism $\zeta^*$ of $\kk(X)^{\times}$ arbitrarily to an automorphism of $(\kk(X)_s)^{\times}$ (which we also call $\zeta^*$).  We let $\zeta$ act on $G_{\kk(X)}$ by 
		$$(\zeta g) (a) = (\zeta^*)^{-1} (g (\zeta^* (a)))$$ 
		for $g \in G_{\kk(X)}$ and $a \in \kk(X)_s^{\times}$. These two actions are compatible, and thus give an action of $\zeta$ on $\HH^i(\kk(X), (\kk(X)_s)^{\times})$ for all $i$.  This action is independent of the extension of $\zeta^*$ chosen, as can be seen by setting $i=0$ and using the fact that cohomology is a universal $\delta$-functor.  Setting $i = 2$ gives an action on $\Br \kk(X)$.	This action clearly preserves $\Br X \subseteq \Br \kk(X)$, thus giving an action of $\zeta$ on $\Br X$.  
	
		\begin{prop}\label{Pringaction}
		The action of $\zeta$ on $\Br X$ extends to an action of the ring $\Z[\zeta]/(1 + \zeta + \cdots + \zeta^{p-1})$ on $\Br X$.
		\end{prop}
		\begin{proof}
		The action of $\Z$ is by scalar multiplication, and this action obviously commutes with the action of $\zeta$.  So it remains 
		to show that $\omega := \sum_{i=0}^{p-1} \zeta^i$ kills $\Br X$.  Now, we have the corestriction and restriction maps
		$$\cores \colon \HH^i(\kk(X), (\kk(X)_s)^{\times}) \to \HH^i(\kk(\PP^2), (\kk(\PP^2)_s)^{\times})$$
		and
		$$\res \colon \HH^i(\kk(\PP^2), (\kk(\PP^2)_s)^{\times}) \to \HH^i(\kk(X), (\kk(X)_s)^{\times}).$$
		We claim that $\res \circ \cores$ is equal to the action of $\omega$ on $\HH^i(\kk(X), (\kk(X)_s)^{\times})$ for all $i$.  
		Since both are endomorphisms of the universal $\delta$-functor, it suffices to show this when $i=0$, in which case both maps are
		simply $\iota \circ \Tr_{\kk(X)/\kk(\PP^2)}$, where $\iota$ is the inclusion $\kk(\PP^2) \hookrightarrow \kk(X)$. Thus the claim is proved.
	
		Let $X_1$ be the set of all codimension $1$ points in $X$, and let $Z_1$ be the set of all codimension $1$ points in $\PP^2$.
		We have a commutative diagram
		\[
		\xymatrix{
					0 \ar[r] & \Br X \ar[r] & \Br \kk(X) \ar[r]^(.35){\partial} \ar[d]^{\cores} & \displaystyle\bigoplus_{x \in X^1} \HH^1(\kk(x), \Q/\Z) \ar[d]^{\cores} \\
					0 \ar[r] & \Br \PP^2 \ar[r] & \Br \kk(\PP^2) \ar[r]^(.35){\partial} & \displaystyle\bigoplus_{z \in Z^1} \HH^1(\kk(z), \Q/\Z) \\
				}
		\]
		where the maps $\partial$ are the residue maps defined in Section \ref{sec:res} and the rows are exact.  The diagram commutes because of the
		naturality of the residue map and the fact that corestriction is a morphism of $\delta$-functors.  Furthermore,
		the diagram shows that if $\alpha \in \Br X
		\subseteq \Br \kk(X)$, then $\cores(\alpha)$ has trivial residues.  Thus $\cores(\alpha) \in \Br \PP^2$, 
		which is trivial.  Since $\omega = \res \circ \cores$, it follows that $\omega$ kills $\alpha$.
		\end{proof}
	
		\begin{remark}
		Proposition \ref{Pringaction} above allows us to think of $\zeta$ as a $p$th root of unity.
		\end{remark}
	
		\begin{cor}\label{Cringstruct}
		We have $\Br X[1-\zeta] \subseteq \Br X[p]$.
		\end{cor}
		\begin{proof}
		This follows from Proposition \ref{Pringaction} and the fact that $1-\zeta$ divides $p$.
		\end{proof}
	
\section{$p$-torsion divisor classes and Brauer elements}\label{sec:pullback}
	In this section, we show how to obtain $(1 - \zeta)$-torsion Brauer elements on $X$ from certain divisor classes on $C$.  The correspondence passes through Brauer elements on $U=\P^2\setminus(C \cup L)$.
	
	\subsection{Divisor classes on $C$}

		We will be concerned with the group $\left(\Pic C/\Z L\right)[p]$.  Since any divisor class in $\Pic C$ is $p$-divisible if its degree is $0$ modulo $p$, and the degree of $C$ is $pd,$ the short exact sequence
		\[
			0 \to \Pic \PP^2 \to \Pic C \to \Pic C/\Z L \to 0
		\]
		together with the Snake Lemma proves that $\left(\Pic C/\Z L\right)[p]$ is an extension
		\[ 
			0	\to \Jac C[p] \to \left(\Pic C/\Z L\right)[p] \to \Z/p \to 0.
		\]

		{The goal of this section is to prove  Theorem~\ref{thm:residueBr}, which we restate here for the reader's convenience.}
		\begin{thm_nonum}[\ref{thm:residueBr}]
		The following diagram commutes and the top row is exact.
		\[
				\xymatrix{
					0 \ar[r] & 
					\frac{\Pic X}{\Z H + (1 - \zeta)\Pic X} \ar[r]^j & 
					\left(\frac{\Pic C}{\Z L}\right)[p] 
					\ar[r]^{\psi} \ar@{^(->}^{\phi}[d] & 
					\Br X[1 - \zeta] \ar[r] \ar@{^(->}[d] & 0 \\
					& & \Br U[p]\ar[r]^{\pi^*}
					&\Br \kk(X)[p]
				}
		\] 
		\end{thm_nonum}
		We begin in Section~\ref{subsec:DefiningPsi} by defining the map $\psi$, and prove that it surjects onto $\Br X[1 - \zeta]$ in Section~\ref{sec:Surjectivity}.  Then we define the map $j$ in Section \ref{subsec:Definingj}.  Finally, in Section \ref{subsec:Proof}, we show that these maps fit together in the above exact sequence.  In Section \ref{subsec:Surjectivity2} we give an alternate proof of surjectivity of $\psi$ when $p = 2$.
		{
		\begin{remark}\label{rmk:Galois}
			If $\pi \colon X \to \PP^2$ and $C$ are defined over a subfield $k_0\subset k$, then all groups in Theorem~\ref{thm:residueBr} have an action of $\Gal(k/k_0)$.  It will be clear that the definitions of $\psi$ and $j$ respect the Galois action.  Thus, we obtain an exact sequence and commutative diagram of $\Gal(k/k_0)$-modules.
		\end{remark}}
		\begin{remark}
			When $p = 2$, $\Br X[1- \zeta] = \Br X[2]$, so we obtain all $2$-torsion Brauer classes.  For $p > 2$, the containment $\Br X[1 - \zeta] \subset \Br X[p]$ will always be strict; this will follow from the argument in Section~\ref{subsec:Surjectivity2}.
		\end{remark}
		
	\subsection{Obtaining Brauer elements on $X$}\label{subsec:DefiningPsi}
		\begin{prop}\label{prop:kerToBr}
			There exists an injective homomorphism $\phi \colon \left(\Pic C/\Z L\right)[p] \to \Br U[p]$ and a homomorphism $\psi\colon \left(\Pic C/\Z L\right)[p] \to \Br X[p]$ such that the diagram
			\[
				\xymatrix{
					\left(\Pic C/\Z L\right)[p] 
					\ar[d]^{\psi} \ar@{^(->}[r]^(.60){\phi}& 
					\Br U[p]\ar[d]^{\pi^*}\\
					\Br X[p]\ar@{^(->}[r] &\Br \kk(X)[p]
				}
			\]
			commutes.
		\end{prop}
		\begin{proof}
			Let $D\in \Div C$ be such that $\overline{[D]}\in \left(\Pic C/\Z L\right)[p]$; then there is a $g_D\in \kk(C)^{\times}$, unique up to constants, and a unique $n_D\in \Z$ such that
			\[
				\divv(g_D) = pD + n_D(L\cap C).
			\]
			We claim:
			\begin{lemma}\label{lem:inj_hom}
				The map $D\mapsto  (g_D, n_D) \in \kk(C)^{\times}/k^\times \times \Z$ induces a well-defined injective homomorphism
			\[
				\phi' \colon \left(\Pic C/\Z L\right)[p] \longrightarrow \kk(C)^{\times}/\kk(C)^{\times p} \times \Z/p.
			\]
			\end{lemma}
			\begin{proof}
				First we note that since $p\neq \Char(k)$ and $k$ is separably closed, we have $k^\times\subset \kk(C)^{\times p}$.  Now let $D, D'$ be two linearly equivalent divisors whose classes are in $\left(\Pic C/\Z L\right)[p]$ and let $h_{D',D}\in \kk(C)$ be such that $\divv(h_{D',D}) = D' - D$.  Then
			\[
				\divv(h_{D',D}^pg_D) = pD' - pD +  pD + n_D(L\cap C),
			\]
			Since $g_{D'}$ and $n_{D'}$ are unique, this implies that $g_{D'} = g_{D}h_{D',D}^p$ and $n_{D'} = n_{D}$.  Thus the map is well-defined on $\Pic C$.
			
			If $D' = D + m(L\cap C)$ for some integer $m\in \Z$, then
			\[
				\divv(g_{D}) = pD + n_D(L\cap C) = pD' + (n_D - mp)(L\cap C).
			\]
			Thus, $g_{D'} = g_D$ and $n_{D'} = n_D - mp$ so we obtain a well-defined map $\phi'$ on $\left(\Pic C/\Z L\right)[p]$.  It is clear from the definition that $\phi'$ is also a homomorphism; it remains to prove injectivity.
			
			Let $D \in \Div C$ be such that $\overline{[D]} \in \left(\Pic C/\Z L\right)[p]$ and such that $g_{D} = h^p$ and $n_{D} = pm$ for some $h\in \kk(C)^{\times}$ and some $m\in \Z$.  Then
			\[
				p\divv(h) = \divv(g_{D}) = pD + n_D(L\cap C),
			\]
			so $D \sim -m(L\cap C)$.  This completes the proof of the lemma.
			\end{proof}
			We resume the proof of Proposition~\ref{prop:kerToBr}.  We claim that the following diagram exists, commutes, and has exact rows.  From this, the proposition follows easily.
			\begin{equation}\label{eq:BigDiagram}
				\xymatrix
				{
					\left(\Pic C/\Z L\right)[p] \;\ar@{^(->}[rr]^{\phi'}
						\ar[rd]^{\phi} \ar[dd]^{\psi}
					& &\frac{\kk(C)^{\times}}{\kk(C)^{\times p}} \times \Z/p \ar@{^(->}[d] \\
					& \Br U[p] \ar@{^(->}[r]^{\partial} \ar[d]^{\pi^*}
					& 
					\frac{\kk(C)^{\times}}{\kk(C)^{\times p}}\times 
					\frac{\kk(L)^{\times}}{\kk(L)^{\times p}}
					\ar[d]^{\pi^*} \ar[r]^{r} 
					&\displaystyle\bigoplus_{P\in C\cup L} \Z/p\\ 
					\Br X[p] \ar@{^(->}[r] & \Br \pi^{-1}(U)[p] \ar[r]^(.45){\partial}
					& \frac{\kk(\tilde C)^{\times}}{\kk(\tilde C)^{\times p}}\times 
					\frac{\kk(\pi^{-1}(L))^{\times}}{\kk(\pi^{-1}(L))^{\times p}}
			}
			\end{equation}
			The existence of the first row follows from Lemma~\ref{lem:inj_hom}.  The existence and exactness of the second and third row follow from~\cite{Grothendieck} and~\cite[Thm. 1]{ArtinMumford} (in the second row, $r$ is the order map).  The  bottom square commutes by functoriality.  It remains to consider the diagonal map, the leftmost vertical map and the top right vertical arrow.
			
			We define the top right arrow componentwise; we take the identity map on the first component and
			\[
				\Z/p\to \kk(L)^{\times}/\kk(L)^{\times p}, \quad
				n\mapsto (\tilde\ell^{dp}/f)^n 
			\]
			on the second. 
			This is injective since $C$ meets $L$ in $pd$ distinct points.  After composing this map with $\phi'$ we obtain
			\[
				\left(\Pic C/\Z L\right)[p] \to \kk(C)^{\times}/\kk(C)^{\times p}\times 
					\kk(L)^{\times}/\kk(L)^{\times p},
					\quad [D]\mapsto (g_D, \tilde\ell^{n_Ddp}f^{-n_D}).
			\]  
			By~\cite[Thm. 1]{ArtinMumford}, 
			
			\[
				\left(r( h_C, h_L) \right)_P = \begin{cases}
					v_P(h_C) + v_P(h_L) & \textup{if }P\in L\cap C,\\
					v_P(h_C) & \textup{if }P \in C\setminus (L\cap C),\\
					v_P(h_L) & \textup{if }P \in L\setminus (L\cap C).\\
				\end{cases}
			\]
			From the definition of $g_D$ and $n_D$, it is clear that $\left(r( h_C, h_L) \right)_P$ is trivial in $\Z/p$ for all points $P\in C\cup L$.  Thus, we obtain a map $\phi \colon \left(\Pic C/\Z L\right)[p] \to \Br U[p]$.
			
			To prove that there exists a map $\psi$ making the diagram commute, it suffices to show that $\partial\circ\pi^*\circ\phi = \pi^*\circ\partial\circ \phi$ is identically $0$.  By the previous paragraph, we know that $\partial\circ\phi \colon D \mapsto (g_D, (\tilde\ell^{dp}/f)^{n_D})$.  Then we observe that $\tilde\ell^{dp}/f$ is a $p^{th}$-power in $\kk(\pi^{-1}(L))$ and that $\pi^*$ acts as raising to the $p^{th}$-power on $\kk(C)^{\times}/\kk(C)^{\times p}$, which completes the proof.
		\end{proof}

	\subsection{Surjectivity of $\psi$}\label{sec:Surjectivity}
		By construction, it is clear that $\im \psi \subseteq \Br X[1 - \zeta]$.  To show the surjectivity of $\psi$ onto $\Br X[1 - \zeta]$, we prove:
		\begin{prop}\label{prop:surj}
			\[
				\Br X[1 - \zeta] = \im \psi.
			\]
		\end{prop}
		\begin{proof}
			This follows immediately from Corollary~\ref{Cringstruct} and the following lemmas, which we prove below.
			\begin{lemma}\label{lem:SurjectionOnInvariants}
				\[
					\Br X[1 - \zeta]\subset 
					\im \pi^* \colon \Br \kk(\PP^2) \to \Br\kk(X).
				\]
		\end{lemma}
		\begin{lemma}\label{lem:Maximality}
			\[
				{\pi^*(\Br \kk(\PP^2)) \cap \Br X =  \im \psi.}
			\]
		\end{lemma}
		\end{proof}
		
		\begin{proof}[Proof of Lemma~\ref{lem:SurjectionOnInvariants}]
			Consider the Hochschild-Serre spectral sequence in group cohomology
			\[
				E_2^{p,q} = \HH^p(\Gal(F'/F), \HH^q(F', F_s^{\times})) \Rightarrow L^{p + q} := \HH^{p+q}(F, F_s^{\times}),
			\]
			where $F'/F$ is a cyclic extension of degree $p$.  By Hilbert's theorem 90 and Tate cohomology, $E_2^{p,q} = 0$ if $q = 1$, or $q = 0$ and $p = 1$ or $p= 3$.  Therefore, $L^2 = \Br F$ surjects onto $\left(\Br F'\right)^{\Gal(F'/F)} = E_2^{0,2} = E_{\infty}^{0,2}$.  Setting $F' = \kk(X)$ and $F = \kk(\PP^2)$, we have a surjection
			\[
				\pi^* \colon \Br \kk(\PP^2) \to (\Br \kk(X))^{\zeta}.
			\]
			We complete the proof by observing that $\Br X[1 - \zeta] \subset (\Br \kk(X))^{\zeta}.$
		\end{proof}
		
		\begin{proof}[Proof of Lemma~\ref{lem:Maximality}]
			From the definition of $\psi$, it is clear that $\im \psi \subset \pi^*(\Br \kk(\PP^2)[p])\cap \Br X[p]$.  Therefore, it remains to prove the reverse containment.
			
			{Let $\calA\in \Br \kk(\PP^2)$ be a Brauer class of exact order $m$ such that $\pi^*\calA\in \Br X$.  Let $Z\subset X$ be an irreducible curve and let $W := \pi(Z)$.  Since $\partial_Z(\pi^*(\calA)) = \pi^*(\partial_W(\calA))$, $\pi^*(\calA)\in\Br X$ if and only if $\partial_W(\calA)$ is in the kernel of the restriction morphism $\HH^1(\kk(W), \Q/\Z) \to \HH^1(\kk(Z), \Q/\Z)$.  However, the kernel of restiction is a $p$-group, thus we may assume that $m = p$.}
			
			Denote the irreducible components of the ramification divisor of $\calA$ that are different from $C$ and $L$ by $W_1, \ldots, W_n$.  Let $g_i$ be a homogeneous polynomial of degree $m_i$ such that $W_i = V(g_i)$.  Since $\pi^*\calA$ is unramified and $W_i\neq C$ or $L$, the residue of $\calA$ at $W_i$ must be equal to $(f/\ell^{dp})^{a_i}$ for some integer $a_i$. Consider the central simple algebra
			\[
				\calB := \calA - \pcyclic{(f/\ell^{dp})^{a_1}}{g_1/\ell^m_1}
				- \pcyclic{(f/\ell^{dp})^{a_2}}{g_2/\ell^m_2}
				- \cdots - \pcyclic{(f/\ell^{dp})^{a_n}}{g_n/\ell^m_n}.
			\]
			From the definition, one can easily see that $\calB \in \Br U$ and that $\pi^*\calA = \pi^*\calB$.  It remains to show that $\calB \in \im \phi$.  Since $\pi^*\calB\in \Br X$, the residue of $\calB$ at $L$ must be equal to $(\tilde\ell^{dp}/f)^{n}$ for some integer $n$.  Additionally, by the exactness of the second row of~\eqref{eq:BigDiagram}, the residue at $C$ must be some function $g$ such that 
			\[
				\divv(g) = pD + n(L\cap C)
			\]
			for some divisor $D$.  Therefore, $\calB = \phi(D)$, which completes the proof.
		\end{proof}

	\subsection{Obtaining elements of $\left(\Pic C/\Z L\right)[p]$ from %
	divisors on $X$}\label{subsec:Definingj}
		To each reduced and irreducible divisor $Z\subset X\setminus\tilde C$, we associate a divisor $D_Z := \pi_*(Z\cap \tilde C)$ in $C$.   Extending the map by linearity, we obtain a homomorphism
		\[
			j' \colon \Div (X\setminus \tilde C) \to \Div C.
		\]
		Let $H = \pi^*L$.
		\begin{lemma}
			The map $j'$ induces a homomorphism
			\[
				j \colon \frac{\Pic X}{\Z H + (1 - \zeta)\Pic X}\longrightarrow
				\left(\Pic C/\Z L\right)[p].
			\]
		\end{lemma}
		\begin{proof}
			First note that $j'$ certainly sends principal divisors on $X$ supported away from $\tilde C$ to principal divisors on $C$.  Then, after composing $j'$ with the quotient $\Div C \to \Pic C$, we may extend the map to $\tilde C$ by finding a linearly equivalent divisor with disjoint support.  Thus, we have a homomorphism from $\Pic X \to \Pic C$.  It is clear from the definition that $j'((1 - \zeta)Z) = 0$, so we may also take the quotient of the domain by $(1 - \zeta)\Pic X$.  Next, we observe that by the projection formula (see, e.g., \cite{Liu}, p.\ 399),
			\[
				L\cap C = L \cap \pi_*(\tilde C) = \pi_*(\pi^*L \cap \tilde C) = \pi_*(H \cap \tilde C).
			\]
			Hence if we quotient $\Pic X$ by $\Z H$ and quotient $\Pic C$ by $\Z L$, the map stays well-defined.
			
			Finally, we show that $\overline{[D_Z]}$ is contained in $\left(\Pic C/\Z L\right)[p]$.  The projection formula implies
			\[
				pD_Z = \pi_*(Z\cap p\tilde C) = \pi_*(Z\cap\pi^*C) = 
				(\pi_*Z)\cap C.
			\]
			However, $\pi_*Z \equiv \deg(\pi_*(Z))\cdot L$.  Thus, $pD_Z - \deg(\pi_*(Z))\cdot (L\cap C)$ is principal on $C$, and so $\overline{[D_Z]}\in \left(\Pic C/\Z L\right)[p]$.
		\end{proof}

	\subsection{Proof of Theorem~\ref{thm:residueBr}}\label{subsec:Proof}
	
			{The commutativity follows from Proposition~\ref{prop:kerToBr} and the surjectivity of $\psi$ follows from Proposition~\ref{prop:surj}.  It remains to prove exactness in the middle and injectivity of $j$.  We do so by proving that
			\[
				\im j = \ker (\pi^*\circ \phi),\quad\textup{ and }\quad
				\ker (\phi\circ j) = 0
			\]
			and then using Proposition~\ref{prop:kerToBr} to conclude that $\im j = \ker \psi$ and that $\ker j = 0$.}  To do this, we first prove a preliminary lemma:
			\begin{lemma}\label{lem:phi'}
				Let $Z = Z_1 - Z_2\in \Div(X)$ with $Z_i$ effective, let $g_i$ be a homogeneous polynomial such that $V(g_i) = \pi_*(Z_i)$, set $n := \deg(g_1) - \deg(g_2)$ and set $g = g_1/g_2$.  Then
				\[
					\phi'(j(Z)) = \left(g/\ell^n, -n\right).
				\]
			\end{lemma}
			\begin{proof}
				Unravelling the definitions, it suffices to show that 
				\[
					\divv(g/\ell^n) \sim p\pi_*((Z_1 - Z_2) \cap \tilde C) - n(L \cap C)
				\]
				as divisors on $C$.   The first term on the right-hand side is equal to $\pi_*((Z_1 - Z_2) \cap \pi^*C)$, which is equal to 
				$\divv(g)$ (as a divisor on $C$) by the projection formula.  The lemma follows immediately.  
			\end{proof}
		
			Let us resume the proof of the theorem.  Let $Z\in\Div(X)$ and let $g, n$ be as in the lemma.  We claim that $\phi(j(Z))$ is equal to the $p$-cyclic algebra $\calA := \pcyclic{\tilde\ell^{dp}/f}{g/\ell^n}$.  To prove this, it suffices to compare the residues of $\calA$ to the residues of $\phi(j(Z))$.  By Theorem~\ref{thm:res}, the residue of $\calA$ at $C$ is equal to $g/\ell^n$, the residue of $\calA$ at $L$ is $\left(f/\tilde\ell^{dp}\right)^n$, and the residue of $\calA$ at $W_i$, a component of $\pi_*(Z)$ of multiplicity $m_i$ is $(\tilde\ell^{dp}/f)^{m_i}$; the residues at all other curves are trivial.  Observe that if $m_i$ is nonzero modulo $p$, then a component of $Z$ must map isomorphically onto $W_i$.  This then implies that $\tilde\ell^{dp}/f$ is a $p^{th}$ power in $\kk(W_i)$.  Thus, all components of $\pi_*(Z)$ have trivial residue.  By Lemma~\ref{lem:phi'} and the proof of Proposition~\ref{prop:kerToBr}, $\phi(j(Z))$ has residue $g/\ell^n$ at $C$, residue $\left(\tilde\ell^{dp}/f\right)^{-n}$ at $L$, and trivial residue elsewhere.   Thus, $\phi(j(Z)) = \calA$, which is plainly in $\ker \pi^*$, so $\im j \subset \ker \pi^*\circ \phi.$ 
			
			Now assume that $\phi(j(Z)) = \calA$ is trivial in $\Br U$.  Then $g/\ell^n = \Norm_{\kk(X)/\kk(\PP^2)}(\gamma)$ for some $\gamma\in \kk(X)^\times$ by Proposition \ref{norm} (iv). Equivalently, $\pi_*Z - nL\in \pi_*\textup{Princ}(X)$.  Therefore, $[Z]\in \Z H + (1 - \zeta)\Pic X$ and $\phi \circ j$ is injective.
			
			Now let $\calB \in \im \phi\cap \ker \pi^* \subset \Br U$.  Since $\pi^*\calB = 0$ and $\kk(X)/\kk(\PP^2)$ is a cyclic extension, $\calB$ must be of the form $\pcyclic{f/\tilde\ell^{dp}}{h}$ for some function $h \in \kk(X)^{\times}$.  Furthermore, by Proposition~\ref{prop:BrauerFacts}, we  may multiply $h$ by a power of $f/\tilde\ell^{dp}$ to assume that $v_C(h) = 0$.  On the other hand, since $\calB\in \im \phi$, for all prime divisors $W\subset \PP^2$ different from $L$, either $v_W(h) \equiv 0 \bmod p$ or $f/\tilde\ell^{dp}\in\kk(W)^{\times p}.$  Thus $\divv(h) = \pi_*(Z) - nL$ for some divisor $Z$ on $X$ and some integer $n$.  This implies that $\calB\in \im j$, which completes the proof.\qed

	\subsection{Alternative proof of surjectivity when $p = 2$} 
	\label{subsec:Surjectivity2}
		When $p = 2$ we can give a brief alternate proof of the surjectivity of $\psi$ with a cardinality argument as follows.  By Corollary~\ref{cor:Brdim},
		\[
			\dim_{\F_2}\Br X[2] = 1 - \rho + (p - 1)\left(1 + 2g(C)\right) = 2 + 2g(C) - \rho.
		\]
		We also have that 
		\[
			\dim_{\F_2}\left(\Pic C/\Z L\right)[2] = 2g(C) + 1, 
			\quad\textup{and}\quad
			\dim_{\F_2}\Pic X/(\Z H + 2\Pic X) = \rho - 1.
		\]
		Then surjectivity follows since $2 + 2g(C) - \rho + \rho - 1 = 2g(C) + 1$.\qed
		{
		\begin{remark}
			The above argument can be strengthened to show that, for odd $p$, $\Br X[1 - \zeta] \subsetneq \Br X[p]$.  To do so, first note that $\dim_{\F_p}\Pic X/(\Z H + (1 - \zeta)\Pic X)$ is always bounded below by $(\rho - 1)/(p - 1)$.  Therefore,
		\[
			\dim_{\F_p}\frac{\Br X[p]}{\Br X [1 - \zeta]} \geq 
			(p-2)\left(1 + 2g(C) - \frac{\rho - 1}{p - 1}\right).
		\] 
		Then, by Proposition~\ref{prop:NumericalInvariants} $p_g = h^{2,0} >0$, so $\rho$ is strictly less than $b_2$, and thus $1 + 2g(C) > \frac{\rho - 1}{p - 1}$. Hence for odd $p$, $\dim_{\F_p}\frac{\Br X[p]}{\Br X (1 - \zeta)}$ is strictly positive.
		\end{remark}}
	
\section{Geometric constructions for Brauer elements in the case $p = 2$} \label{sec:theta}

		Henceforth we restrict to the case that $p = 2$.  We allow $C$ to be a curve of arbitrary degree $e$, not necessarily even.  

		In this section we will describe a method that, given an element of $(\Pic C/\Z L)[2]$, constructs a class in $\Br \kk(\P^2)[2]$. 
		In the case that $e = 2d$ is even, our construction even yields a class in $\Br X[2]$ where $\pi \colon X \rightarrow \P^2$ is the double cover
ramified on $C$.  In any case, the Brauer class on $\kk(\PP^2)$ will be ramified only on $C$ and possibly $L$, thus giving an element of $\Br U[2].$
	
		The method works over any ground field of characteristic prime to $2$ over which the relevant objects are defined.  Thus, throughout, we will let $\calL$ be a line bundle on $C$ whose class in $\Pic C/\Z L$ is $2$-torsion, and let $k_0$ be any field over which $C$ and $\calL$ are defined
		and such that $\Char(k_0) \neq 2$.  Let $k$ denote a separable closure of $k_0$; this is compatible with our previous notation.

\subsection{Symmetric resolution on line bundles}\label{Ssymm}
To begin the construction we form a symmetric resolution of $\calL$ by line
bundles on $\P^2$.

Since the class of $\L$ in $\Pic C/\Z L$ is $2$-torsion, $\L \otimes \L \simeq
O_C(\e L)$ where $\e$ is some integer.  
By twisting $\L,$ we can
and do assume that $\e =0$ or $1$. 
We will use the following result of Catanese~\cite{Cat}, which is also
described in \cite{Cat97}.  The theorem is stated for algebraically closed 
fields of characteristic 0 in the above citations,  but the proof works for any field of characteristic not 2.
\begin{thm}[{\cite[Prop 2.28]{Cat}}]\label{thm:Catanese}
Let $C$ be a smooth curve in $\P^2$ and let $\L$ be a line bundle on
$C$ such that $\L \otimes \L \simeq \O_C(\e)$ where $\e=0$ or $1$. 
 Then there is a
symmetric resolution of $\L$ as a sheaf on $\P^2$ by line bundles.
More precisely we have:
$$ \xymatrix{0 \rightarrow \displaystyle\bigoplus_{i=1}^n \O_{\P^2}(a_i - e + \e) \stackrel{M}{\longrightarrow}
\displaystyle\bigoplus_{i=1}^n \O_{\P^2}(-a_i) \rightarrow \L \rightarrow 0}$$
where $M=(m_{ij})$ is a $n\times n$ symmetric matrix with
$\deg(m_{ij}) = e - a_i - a_j - \e$ and $a_i \geq 0$.
\end{thm}

\begin{cor}
	We have that $V(\det M) = C$ and the degree sequence $(d_i)=(\deg(m_{ii}))_{i = 1}^n$ gives a partition of $e$ into $n$ parts, each of which has the same parity.
\end{cor}
\begin{proof}
	This is an immediate consequence of Catanese's result.
\end{proof}

\begin{proposition}\label{Pdimcount}

Let $\mathbb L$ be the moduli space of pairs $(C,\L)$ of smooth planar
  curves $C$ of degree $e$ with $\L \in
  (\Pic C/\Z L)[2]$ such that the symmetric resolution of $\L$
  yields a fixed partition $e=d_1+\cdots+d_n$.  Recall that $a_i=(e-d_i-\e)/2$.
Then
\[
\dim\mathbb{L} =  \sum_{1 \leq i \leq j \leq n} \binom{e-a_i-a_j-\e+2}{2} - \sum_{1
  \leq i , j \leq n} \binom{a_i-a_j+2}{2} -8.
  \]
The partitions $e=1+\cdots+1$,
$e=2+\cdots+2$, $e=3+1+\cdots+1$, and the trivial partition $e=e,$ occur for
generic curves $C$.  All other partitions only occur for special planar
curves of degree $e$.

\end{proposition}
\begin{proof}
Let $V =\bigoplus \OO_{\P^2}(-a_i)$ and 
let $p \colon \P_{\P^2}(V) \rightarrow \P^2$ be the projection.
Let $\Aut_{\P^2}\P(V) =\{ \psi \in \Aut \P(V) \st p \circ \psi = p \}$
be the automorphisms that act fibrewise. 
The first statement comes from counting parameters for the matrices
$M$ and subtracting the dimension of automorphisms $\dim \Aut\P_{\P^2}(V) =
\dim \Aut_{\P^2}(\P(V)) \rtimes \Aut \P^2$.  That is,
\[
 \dim \mathbb{L} = \sum_{1 \leq i \leq j \leq n} \hh^0(\O(\deg (m_{ij}))) -1  - \dim \Aut_{\P^2}\P(V) -\dim \Aut \P^2.
 \]
Since $\mathbb{L}$ is finite
over planar curves and smooth planar curves of degree $\geq 4$ have
finite automorphism groups, this yields the dimension of the moduli space.

Let $|C|$ be the dimension of the moduli space of plane curves of degree $e$.  By Proposition~\ref{Pcombinatorics},
\[
	\dim \mathbb{L} \leq \dim |C| - \dim \Aut \PP^2 = (\binom{e+2}{2} - 1) - 8,
\]
with equality if and only if the partition is $e=1+\cdots+1$,
$e=2+\cdots+2$, $e=3+1+\cdots+1$, or the trivial partition $e=e.$ 
\end{proof}

\subsection{Constructing Brauer classes}\label{subsec:scrAmaps}
{Henceforth, we will consider a fixed nontrivial line bundle $\calL$ such that $\L \otimes \L \simeq \O_C(\e)$ for $\e\in\{0,1\}$, and then let  $M$ and $a_i$ be as in Theorem~\ref{thm:Catanese}.  Since $\calL$ is nontrivial, $n\geq 2$.  We note that $\e, M, a_i$ depend on the choice of $\calL,$ despite the lack of dependence in the notation.  {After possibly reordering, we will assume that $\deg{m_{ii}}$ is a weakly decreasing sequence}}

	Write 
	\[
		V=\bigoplus_{i=1}^n \O_{\P^2}(a_i-e+\e).
	\]
	Since $M \colon V \rightarrow V^* (\e - e)$ is symmetric, we can interpret $M$ as a quadratic form $M \colon \Sym V \rightarrow \O_{\P^2}(\e-e)$.  Equivalently, we may consider the subvariety $V(x^tMx) \subset \P_{\P^2}(V)$ which gives a
quadric bundle over $\P^2$.  
\begin{example}\label{ex:GenericPartitions}\ \\
	\vspace{-2mm}
	\begin{enumerate}
	\item If $M$ corresponds to a partition into $n$ equal parts, then the subvariety is a hypersurface of degree $(2, e/n)$ in $\P^{n-1} \times \P^2$.  Projection onto the second factor gives the quadric bundle over $\PP^2$.  
	\item The partition $e = 3 + 1 + \cdots + 1$ yields a matrix $M$ with degrees
\[
	\begin{pmatrix} 
		3 & 2 & 2& \cdots & 2 \\
		2 & 1 & 1 & \cdots & 1 \\
		2 & 1 & 1 & \cdots & 1 \\
		\vdots & \vdots & \vdots &\ddots& \vdots \\
		2 & 1 & 1 & \cdots & 1 
	\end{pmatrix}
\]
Let $x,y,z$ be coordinates on our base $\P^2$ and let $Y\subset \PP^{e-1}$ be the cubic $(e-2)$-fold cut out by the vanishing of $v^tMv$ where $v = (1,v_0,\dots,v_{e-4})$ and $\P^{e-1}$ has coordinates $v_0,\dots,v_{e-4},x,y,z$.  Projection from the linear space $\Pi = V(v_0, \ldots, v_{e-4})$ gives rise to a quadric bundle
\[
	\Bl_\Pi Y \to \PP^2.
\]
	\end{enumerate}
\end{example}	

If $\e-e$ is even we can twist by $(\e - e)/2$ to obtain a quadratic form that
takes values in $\O_{\P^2}$.  If $\e - e$ is odd, then we may twist by $(\e - e - 1)/2$ and trivialize $\OO_{\P^2}(1)$ over $\A^2 := \PP^2\setminus L$ to obtain a quadratic form with values in $\O_{\A^2}$.  In both cases, the quadratic form is nondegenerate away from $C$.

	Given a quadratic form $M$ valued in $\OO_{\PP^2}$ or $\OO_{\A^2}$, we may consider the sheaf of Clifford algebras $\mathscr{C}l(M)$ and the sheaf of even Clifford algebras $\mathscr{C}l_0(M)$.  If $n$ is even, then $\mathscr{C}l(M)$ defines an element in $\Br U \subset \Br \kk(\PP^2)$.  If $n$ is odd, then  $\mathscr{C}l_0(M)$ defines an element in $\Br U \subset \Br \kk(\PP^2)$.  In order to avoid making different statements in the case that $n$ is odd or even, we define:
	\begin{defn}
		\[
			\mathscr{A}_U(M) := \begin{cases}
				\mathscr{C}l(M) & \textup{if }n\textup{ is even},\\
				\mathscr{C}l_0(M) & \textup{if }n\textup{ is odd},
			\end{cases}
			\quad\textup{ and, if }e\textup{ is even,}\quad
			\mathscr{A}_X(M) := \begin{cases}
				\mathscr{C}l_0(M) & \textup{if }n\textup{ is even},\\
				\mathscr{C}l(M) & \textup{if }n\textup{ is odd}.
			\end{cases}
		\]
	\end{defn}
	Note that if $e$ is even (so $X$ is defined), then $\pi^*\mathscr{A}_U(M)$ is equal to $\mathscr{A}_X(M)$ in $\Br \kk(X)$~\citelist{\cite{Lam}*{V.2.4} \cite{KnusQF}*{Chap. 4, Lemma 9}}; in fact, if $n$ is odd then $\mathscr{A}_U(M)\otimes_{\kk(\PP^2)}\kk(X)$ is isomorphic to $\mathscr{A}_X(M)$.  
	In particular, this implies that $\pi^*\mathscr{A}_U(M)\in \Br X.$  In summary, we have proved the following proposition.
	\begin{prop}\label{prop:CompatibilityAXAU}
		Assume that $e = 2d$.  Then we have maps 
		\[
			\mathscr{A}_U\colon (\Pic C/\Z L)[2] \to \Br U[2] 
			\quad\textup{ and }\quad
			\mathscr{A}_X\colon (\Pic C/\Z L)[2] \to \Br X[2] 
		\]
		such that the diagram 
		\[
				\xymatrix{
					\left(\Pic C/\Z L\right)[2] 
					\ar[d]^{\mathscr{A}_X} \ar[r]^(.60){\mathscr{A}_U}& 
					\Br U[2]\ar[d]^{\pi^*}\\
					\Br X[2]\ar@{^(->}[r] &\Br \kk(X)[2]
				}
		\]
		commutes.
	\end{prop}
	We note that thus far we have just proved there is a map of \emph{sets}.  In Section~\ref{sec:main}, we will prove that $\mathscr{A}_U$ and $\mathscr{A}_X$ are group homomorphisms.
	\begin{example}\label{ex:paritycases}\ \\
		\begin{enumerate}
			\item $\deg C \textup{ even}, \e = 0$: In this case the set of divisors $\calL$ is exactly $\Jac(C)[2]$.  All $2$-torsion line bundles on generic curves $C$ will have a resolution giving rise to the partition $e = 2 + 2 + \cdots + 2$.  Other partitions of $e$ into even parts will only occur for special curves $C$.
			
			\item $\deg C \textup{ even}, \e = 1$: The set of such line bundles is in bijection with theta characteristics on $C$; the correspondence is twisting by $\frac12(e - 4)$.  On a generic curve $C$, such a line bundle $\calL$ will have a resolution giving rise to the partition $e = 1 + 1 + \cdots + 1$ or $e = 3 + 1 + \cdots + 1$, depending on whether the corresponding theta characteristic is even or odd, respectively.  {Recall that a theta characteristic $\calL$ is even or odd depending on the parity of $\hh^0(\calL)$.  }Other partitions of $e$ into odd parts will only occur for special curves $C$.
			
			\item $\deg C \textup{ odd}, \e = 0$: These line bundles are again in bijection with theta characteristics on $C$, although this time the bijection is given by twisting by $\frac12(e - 4)$.  On a generic curve $C$, such a line bundle $\calL$ will have a resolution giving rise to the partition $e = 1 + 1 + \cdots + 1$ or $e = 3 + 1 + \cdots + 1$, depending on whether the corresponding theta characteristic is even or odd respectively.  As above, other partitions of $e$ into odd parts will only occur for special curves $C$.
		\end{enumerate}
		
	\end{example}
	Note that if $\L\otimes\L \isom\OO_C(\e L)$ then $\deg\L = \frac12e\e$.  Thus, Example~\ref{ex:paritycases} covers all possibilities.

\section{Compatibility of $\psi$ and $\calA_X$}\label{sec:main}

	In this section we compare the maps 
	\[
		\phi\colon (\Pic C/\Z L)[2] \to \Br U[2]
		\quad \textup{and}\quad
		\psi\colon (\Pic C/\Z L)[2] \to \Br X[2]
	\]
	defined in Section~\ref{subsec:DefiningPsi} with the maps 
	\[
		\scrA_U\colon (\Pic C/\Z L)[2] \to \Br U[2]
		\quad\textup{and}\quad
		\scrA_X\colon (\Pic C/\Z L)[2] \to \Br X[2]
	\]
	defined in Section~\ref{subsec:scrAmaps}.  More precisely, we will prove the following theorem.
	\begin{theorem}\label{thm:Compatibility}
		For all $D\in (\Pic C/\Z L)[2]$, we have
		\[
			\phi(D) = \scrA_U(D)_k,
			\quad \textup{ and } \quad
			\psi(D) = \scrA_X(D)_k.
		\]
	\end{theorem}
	{Theorem~\ref{thm:geometricBr} then follows immediately from Theorems~\ref{thm:residueBr} and~\ref{thm:Compatibility}, Proposition~\ref{Pdimcount}, and Example~\ref{ex:GenericPartitions}. 
	
	We will prove Theorem~\ref{thm:Compatibility} by comparing residues of $\phi(D)$ and $\scrA_U(D)_k$.  We begin in Section~\ref{sec:residues} by computing the residues of $\scrA_U(D)$ in terms of minors of the matrix $M = M_D$ associated to $D$ via Theorem~\ref{thm:Catanese}.  Next, in Section~\ref{sec:SymMatrices} we prove some preliminary relations among the minors of symmetric matrices which we use in Section~\ref{sec:WeilDivisor} to relate the vanishing of certain minors of $M_D$ to the Weil divisor $D$.  Finally, we combine everything in Section~\ref{sec:Proof} to complete the proof of Theorem~\ref{thm:Compatibility}.
	
	}

	As before, if $M$ is an $n \times n$ matrix, we let $M_{ij}$ denote the $(n-1) \times (n-1)$ minor obtained by deleting the $i^{\mathrm{th}}$ row and $j^{\mathrm{th}}$ column of $M$.  Furthermore, we let $M_i$ denote the upper-left $i \times i$ principal minor.

	\subsection{Residues of $\mathscr{A}_U(\calL)$}\label{sec:residues}

		In this section, we will compute the ramification data $\mathscr{A}_U(\calL)$.  		
		Let $M^{(0)}$ be the matrix over $\kk(\PP^2)$ obtained from $M$ by dividing every entry by an appropriate power of $\ell$.  We will first perform a change of basis so that we may assume certain minors are nonzero.
		\begin{lemma} \label{lem:smoothseq}
		Let $M$ be a symmetric $n \times n$ matrix over the polynomial ring $k_0[x,y]$, and suppose that $\det M \neq 0$.  Then there is a basis of $k_0[x,y]^n$ so that if we write $M$ in terms of this basis, all the principal minors $M_i$ are non-zero.
		\end{lemma}
		\begin{proof}
			We will first do this over a point in $\A^2$.  Write $Q$ for the corresponding quadric.  We need to find a sequence of smooth hyperplane sections $Q \cap H_1 \subset Q \cap H_1 \cap H_2 \subset \cdots \subset Q \subset H_1 \cap \cdots \cap H_{n-2}$ which are all smooth.  A hyperplane section of a smooth quadric is singular if and only if the hyperplane is a tangent plane.  The generic hyperplane is not a tangent plane, even over finite fields since the number of points in a quadric is smaller than the number of possible hyperplane sections.

Now to prove this generally, we merely do the above change of basis at a point
and then lift to a change of basis over $k_0[x,y]$ via the surjection
$\GL_n k_0[x,y] \twoheadrightarrow \GL_n k_0.$
		\end{proof}			
		\begin{prop}\label{prop:AUResidues}
			The Azumaya algebra $\scrA_U(\calL)$ is unramified away from $C$ and $L$.  At $C$ and $L$, we have
			\[
				\partial_C(\scrA_U(\calL)) = M_{nn}/\ell^{2d - d_n}, 
				\quad \textup{ and }\quad
				\partial_L(\scrA_U(\calL)) = (f/\tilde\ell^{2d})^i, 
			\]
			where $\divv(M_{nn}/\ell^{e - d_n}) = 2D + i(C\cap L)$ for some divisor $D$.
		\end{prop}
		\begin{proof}
			The Azumaya algebra $\scrA_U(\calL)$ is constructed as the Clifford algebra or even Clifford algebra of a smooth quadric bundle over $U$, and as such, is unramified over all of $U$.  Furthermore, by Proposition~\ref{prop:CompatibilityAXAU}, $\pi^*\scrA_U(\calA)$ is unramified on $X$.  Since $\kk(\pi^{-1}(L)) = \kk(L)(\sqrt{f/\tilde\ell^{2d}})$, the residue at $L$ must be $(f/\tilde\ell^{2d})^i$ for some integer $i$.
			
			We now turn to computing the residue at $C$.  By Lemma~\ref{lem:smoothseq}, we may assume that $M_{j}, M_{j}^{(0)}\neq 0$ for all $j$.  Since $\deg(M_{n}) > \deg(M_{j})$ for all $j < n$ and $C$ is reduced and irreducible, $v_C(M_j) = 0$ for all $j < n$ and similarly $v_C(\ell) = 0$ 
Hence, by Proposition~\ref{Presvals}, $\partial_C(\scrA_U(\calL)) = M_{nn}/\ell^{e - d_n}$.  Finally, the compatibility between $i, M_{nn},$ and $d_n$ follows from~\eqref{eq:BigDiagram}.
		\end{proof}

\begin{remark}  One can use Schur's identity to show that  $M_{i+1}/M_{i-1}$
is a square modulo $M_i$ giving an alternate approach to showing that
$\Cl(Q)$ is unramified away from $C$ and $L$.
\end{remark}

	\subsection{Symmetric matrices of homogeneous forms on $\PP^2$}
	\label{sec:SymMatrices}
		Throughout this section $M = (m_{i,j})$ will be an arbitrary symmetric $n\times n$ matrix of homogeneous forms on $\PP^2$ such that $2\deg(m_{ij}) = \deg(m_{ii}) + \deg(m_{jj}).$  Also, $C = V(\det M) \subseteq \P^2$ (note that this is a more general situation than in Section \ref{sec:theta}, as
		we do not assume anything about $C$).
		
	\begin{lemma}\label{lemma:smoothC}
	The curve $C$ is smooth if and only if the rank of $M$ is $n-1$ along $C$. 
	More precisely, if the rank at a point $P\in C$ is $r$, then $C$ has a singularity of multiplicity $n-r$ at that point.
	\end{lemma}

	\begin{proof}
	  Since this is a geometric condition we can assume that we are over an
	algebraically closed field $\overline{k}_0$, and we can pass to the completion at a point $P\in C$ so we will work over $\overline{k}_0 \llbracket x,y \rrbracket$ with maximal
	ideal $\mm=(x,y)$.  Let $M(P)$ denote the matrix $M$ at $P$ and let $r$ denote the rank of $M(P)$.  We can change basis at $P$ so that  $M(P) = 0_{n-r} \oplus I_r$ where $I_r$ is
	the $r \times r$ identity matrix and $0_{n-r}$ is the $(n-r) \times (n-r)$
	zero matrix.
	Now when we lift our chosen basis at $p$ by Nakayama's
	lemma, we get that
	$M=M(p)+xM^x+yM^y+m^2\overline{F}\llbracket x,y \rrbracket ^{n \times n}$ for some symmetric matrices
	$M^x,M^y$ over $\overline{F} \llbracket x,y \rrbracket$.  Note
	that $\det M \in (x,y)^{n-r}$.  So $C=V(\det M)$ is smooth if and only
	if the corank $n-r=1$.  
	\end{proof}

	\begin{remark} In higher dimensions, the discriminant of the quadric
	  bundle will have singularities as long as the map parametrizing the quadrics has dimension $\geq 2$.  It would be interesting to extend the results of this paper to higher dimensions, but this would require an analysis of the behaviour of the singularities of $C=V(\det M_{nn})$
\end{remark}

	\begin{prop}\label{Ptangency}
		Let $M$ be a symmetric $n\times n$ matrix of homogeneous forms on $\P^2.$
		Assume that $C = V(\det M)$ is smooth. Let $j\in \{1, \ldots, n\}.$  Then
		\[
			V(\det M, M_{jj})_{\textup{red}} = 
			V(M_{1j},\ldots,M_{nj})_{\textup{red}}.
		\]
		Further, for every $P\in C$, $v_P(M_{jj}) = 2\min_i\{v_P(M_{ij})\}$.
	\end{prop}
	 \begin{proof}
		By a permuation of the coordinates we may assume that $j = n$. Further, this is a geometric statement, so we will work over the algebraic closure $\overline{k}_0$.  
		
		We will first show the set-theoretic equality $V(\det M, M_{nn})_{\textup{red}} = V(M_{1n},\ldots,M_{nn})_{\textup{red}}$.  The containment $V(M_{1n},\ldots,M_{nn}) \subset V(\det M, M_{nn})$ is clear.  To prove the reverse containment let $P\in V(\det M, M_{nn})_\textup{red}$ and consider the quadric bundle $\mathcal{Q}=V(x^tMx) \subset \P_U^{n-1}$ over $U$. We write $e_1, \ldots, e_n$ for the coordinates on $\P_U^{n-1}$. 
			
			Since $P\in V(\det M, M_{nn})_\textup{red}$, the fiber $Q:=\mathcal{Q}_P$ is singular and the intersection $Q\cap V(e_{n})$ is also singular.  By Lemma~\ref{lemma:smoothC}, $M(P)$ has rank $n - 1$ so $Q$ is the cone of a smooth quadric over a point $v\in \PP^{n-1}$.  Then, since a hyperplane section of $Q$ is singular if and only if it contains the vertex point, the last coordinate of $v$ must be $0$.  As the vertex point of $Q = \mathcal{Q}_P$ can also be viewed as an element of $\ker M(P)$, this implies that $M_{in}$ vanishes at $P$ for $i = 1, \ldots, n$.  Thus 
		\[
		 	V(\det M, M_{nn})_{\textup{red}} = V(M_{1n},\ldots,M_{nn})_{\textup{red}}.
		\]

		To prove the second claim concerning the valuations, 
we pass to the completion at $P$ and work over $\overline{k}_0 \llbracket x,y \rrbracket$ with maximal
ideal $\mm=(x,y)$.  We change coordinates on $\PP^n$ such that $V(e_{n})$ stays fixed and $v = (1:0:\cdots:0)$.  Then we may write 
		\[
			M = 
			\left(\begin{array}{c|ccc}
			    0 & 0 & \cdots & 0 \\ \hline
			    0  & \multicolumn{3}{c}{\multirow{3}{*}{\raisebox{-3mm}{\scalebox{1.5}{$B_0$}}}} \\
				{\vdots} & & &\\
			    0 & & & 
			  \end{array}\right) + 
			\left(\begin{array}{c|ccc}
			    a & v_2 & \cdots & v_n \\ \hline
			    v_2  & \multicolumn{3}{c}{\multirow{3}{*}{\raisebox{-3mm}{\scalebox{1.5}{$B$}}}} \\
				{\vdots} & & &\\
			    v_n & & & 
			  \end{array}\right),
		\]
		where $a, v_i$ and the entries of $B$ are all in $\mm$, and $B_0\in M_{n-1}(\overline{k}_0)$.  Therefore
		\[
			\det M \equiv a \det(B_0) \bmod \mm^2.
		\]
		Since $C$ is smooth, $\det(B_0)$ must be a unit.  Now we pass to the complete discrete valuation ring $\OO_{C, P}\isom \overline{k}_0\llbracket x\rrbracket$.  In this ring, $B_0$ is still invertible, therefore, by the Gram-Schmidt process, we may diagonalize $B_0 + B$ over $\overline{k}_0\llbracket x\rrbracket$ while still preserving $V(e_n)$.  Hence, we may assume that $M$ is of the form
		\[
			\begin{pmatrix}
				a & v_2 & v_3 &\cdots & v_n\\
				v_2 & u_2 & 0 & \cdots & 0\\
				v_3 & 0 & u_3 & \cdots & 0\\
				\vdots & \vdots & \vdots & \ddots & \vdots\\
				v_n & 0 & 0 &\cdots & 1
			\end{pmatrix},
		\]
		where $a, v_i \in x \overline{k}_0\llbracket x\rrbracket$ and $u_i\in  \overline{k}_0\llbracket x\rrbracket^{\times}.$
		Then we have
		\begin{align*}
			\det M & = au_2\cdots u_{n-1} - v_2^2 \prod^{n-1}_{j \neq 2}u_j -v_3^2 \prod^{n-1}_{j \neq 3}u_j - \cdots - v_{n-1}^2 \prod^{n-1}_{j \neq n-1}u_j - v_n^2 \prod^{n-1}_{j =1}u_j\\
			M_{nn} & = au_2\cdots u_{n-1} - v_2^2 \prod^{n-1}_{j \neq 2}u_j -v_3^2 \prod^{n-1}_{j \neq 3}u_j - \cdots - v_{n-1}^2 \prod^{n-1}_{j \neq n-1}u_j\\
			M_{in} & = \pm v_nv_i \prod_{j \neq i}u_j, \quad i = 2, \ldots, n-1\\
			M_{1n} & = \pm v_n\prod_{j}u_j
		\end{align*}
		In $\OO_C$, $\det M = 0$ so $M_{nn} = M_{nn} - \det M = v_n^2\prod_{j=1}^{n-1}u_n$.
		Hence, 
		\[
			v_P(M_{nn}) = v_P(-u_1\cdots u_{n-1}v_n^2) = 2v_P(v_n)
		\]
		 and $\min_iv_P(M_{in}) = v_P(M_{1n})= v_P(v_n)$, which completes the proof.
	\end{proof}
	\subsection{Determining the Weil divisor corresponding to $\calL$}
	\label{sec:WeilDivisor}
		
		To determine the divisor $D$ such that $\L \simeq \O_C(D),$ we will twist $\L$ by some $\O_{\P^2}(i)$ to obtain a line bundle $\mathcal N$ which is effective and has a summand $\O_{\P^2}$ in its resolution.  Then we will apply the general statement below.
		\begin{proposition}  
			Let $\mathcal{N}$ be an effective line bundle of $C$ that has a presentation of the form
			\[
				0\rightarrow \bigoplus_{i=1}^n \O(-\beta_i)
		  	  	\stackrel{M}{\rightarrow} 
				\left(\bigoplus_{i=1}^{n-1} \O(-\alpha_i)\right) \oplus \O 
				\rightarrow \mathcal N \rightarrow 0.
			\]
			Let $e_n$ denote the section of $\mathcal N$ corresponding to $1\in \O$ in the last direct summand.  Then $e_n$ vanishes on the	$(n-1)\times (n-1)$ minors $M_{ni}$ obtained by removing the bottom row	of $M$ and one of the columns.  In particular, $\mathcal N \simeq \O_C(D)$ where $D = V(M_{n1},\ldots,M_{nn})\subset C$.
		\end{proposition}
		\begin{proof}
		Consider $\mathcal N$ as given by this presentation.  We will write $e_n$ as the usual standard basis vector. The section $e_n$ will vanish exactly when $e_n \in \im M$.  This occurs when we can solve the equation $Mx =e_n$.  The rank of $M$ is always at least $n-1$ by Lemma~\ref{lemma:smoothC}, and is exactly	$n-1$ on $C$.  Then $Mx=e_n$ has a solution at a point in $C$ if and only if the rank of the $n	\times (n+1)$ matrix $(M | e_n)$ equals the rank of $M$ which is $n-1$.  Hence we need the locus where $\rank (M|e_n) \leq n-1$, which is exactly given by the vanishing of all the minors $M_{n1},\ldots,M_{nn}$.
		\end{proof}
		{\begin{cor}\label{cor:WeilDivisorOfL}
			$\mathcal{L}\simeq \O_C(V(M_{1n}, \ldots, M_{nn}) - \frac12(e - d_n - \e)L)$.
		\end{cor}}
		\begin{proof}
			{We apply the Proposition to $\calL\otimes\OO_{C}(a_nL)$ and note that $2a_n = e - d_n - \e$.}
		\end{proof}
		\begin{example}
			Consider the resolution
			\[
				 0 \rightarrow \O_{\P^2}(-2) \stackrel{M}{\rightarrow}
				 \O_{\P^2}(-1)^3 \rightarrow \L \rightarrow 0,
			\]
			where
			\[
				M = \begin{pmatrix} 
				ax & bz & by \\
				bz & ay & bx \\
				by & bx & az
				\end{pmatrix}.
			\]
			For generic $a,b$, the curve $C = V(\det M) = V((a^3 +2b^3) xyz -ab^2(x^3+y^3+z^3))$ is a smooth genus $1$ curve.  The conic $E = V(M_{33}) = V(a^2xy-b^2z^2)$ is tangent to $C$ at 3 points and the radical of the ideal $(M_{33},\det M)$ is $(M_{13},M_{23},M_{33})$.  The primary decomposition of this ideal gives the point $P_1=[b : b : a]$ and $P_2 \cup P_3 = V(a(x+y)+bz,a^2y+abyz+b^2z)$.  Hence $\L(1)=\O_C(P_1+P_2+P_3)$.  To find the divisor which represents $\L$,  we simply subtract a line $L$ which we can choose to be $L=V(a(x+y)+bz)$.  This removes $P_2 \cup P_3$ and subtracts the third point $P_4 =[1:-1 : 0]$ in $L \cap C$. Hence $\L=\O_C(P_1-P_4)$.
		\end{example}

	\subsection{Proof of Theorem~\ref{thm:Compatibility}}\label{sec:Proof}
		By Propositions~\ref{prop:kerToBr} and~\ref{prop:CompatibilityAXAU}, if $\phi$ and $\scrA_{U}$ agree on $(\Pic C/\Z L)[2]$ then so do $\psi$ and $\scrA_X$.  Thus, it suffices to show that $\phi(D) = \scrA_{U}(D)_{k}$, which, by~\eqref{eq:BigDiagram}, we may do so by comparing residues.  Since $\phi(D), \scrA_{U}(D)_{k}$ are contained in $\Br U$, it suffices to show that $\partial_C(\phi(D)) = \partial_C(\scrA_U(D)_k)$ as~\eqref{eq:BigDiagram} then implies that the residues agree at $L$.
	
		Fix a divisor $D$ which represents an element of $(\Pic C/\Z L)$.  Without loss of generality, we may assume that $2D \sim \e L$ where $\e = 0$ or $1$.  By construction, 
		\[
			\divv(\partial_C(\phi(D))) = 2D - \e(L\cap C).
		\]
		Combining Propositions~\ref{prop:AUResidues},~\ref{Ptangency} and Corollary~\ref{cor:WeilDivisorOfL}, we have
		{\begin{align*}
			\divv(\partial_C(\scrA_U(D)_k)) & =  \divv(M_{nn}/\ell^{e - d_n})\\
			& =  2V(M_{1n}, \ldots, M_{nn}) + (d_n - e)(L\cap C)\\
			& = 2(D + \frac12(e - d_n - \e)(L\cap C)) + (d_n - e)(L \cap C)\\
			& = 2D -\e(L\cap C).
		 \end{align*}}
		 Since $k$ is a separably closed field of characteristic not $2$, if the divisors of residues of $2$-torsion classes agree, then the residues agree.  This completes the proof. \qed
	
	\begin{remark}
		If we fix a partition $d = \sum d_i$, we can look at the scheme of quotients of the vector bundle $V = \bigoplus \O(-a_i)$ with fixed Hilbert polynomial as determined by the resolution corresponding to the partition $d$.  There will be a subscheme of this Quot scheme that is described by symmetric matrices, and these will determine line bundles on curves.  We can form an open subscheme by demanding that the curves be smooth.
		Conversely, the set symmetric matrices with degree determined by the partition forms an open subvariety in the complete linear system of divisors $\P(V)$ of a fixed linear equivalence class.  There will be an algebraic map from the variety of matrices to the scheme of line bundles.  This map will be equivariant for the action of $\Aut(\P^2)$ and will be defined over our basefield $k_0$.  Catanese's result says that this map is surjective over $k_0$ points.
		We remark that it would be interesting to use the analysis in this paper to consruct a scheme structure on the 2-torsion Brauer group of $X$.
	\end{remark}

\section{Degree $2$ K3 surfaces}\label{sec:examples}

	We now specialize to the case where $C$ is a smooth sextic curve, and $X$ is a double cover of $\PP^2$ branched over this curve, i.e., a degree two K3 surface. We will study the possible symmetric resolutions of the elements of $(\Pic C/\Z L)[2]$.  In particular, we will show that if there exists a resolution of an element of $(\Pic C/\Z L)[2]$ that gives rise to a partition that is different from $1 + 1 + 1 + 1 + 1 + 1$, $2 + 2 + 2$, $ 3 + 1 + 1 + 1$, then $X$ has Picard rank greater than $1$.  This proves that Theorem~\ref{thm:geometricBr} implies van Geemen's result.
	
	  {We note that the analysis for the generic partitions has been done previously, including by Beauville~\cite{Beauville-Determinantal} and van Geemen~\cite{vanGeemen-Degree2}.}  The resolutions corresponding to partitions $3 + 1 + 1 + 1$ and $2 + 2 + 2$ have been used for arithmetic implications in~\cite{HVV} and~\cite{HassettVA} respectively.

We have the following possible partitions of 6 into only even or odd parts.
\begin{center}\begin{tabular}{ccccc}
	partition & parameter count & $\L\otimes \L$ & quadric bundle & Section \\
	6 		& 19 & $\O_C$ & $\P^2$ & \ref{sec:trivial}\\
	4 + 2 	& 18 & $\O_C$ & $X$ & \ref{sec:2plus4}\\
2+2+2 & 19 & $\O_C$ & $(2,2) \subset \P^2 \times \P^2$ & \ref{sec:all2s}\\
5 + 1 & 18 & $\O_C(1)$ & $X$ & \ref{sec:1plus5}\\
3 + 3 & 18 & $\O_C(1)$ & $X$ & \ref{sec:3plus3}\\
3 + 1 + 1 + 1 & 19 & $\O_C(1)$  & $\Bl_\Pi X \subset \P^5 \quad \deg X =3$ & \ref{sec:cubic4fold}\\
1+1+1+1+1+1 & 19  & $\O_C(1)$ & $(2,1) \subset \P^5 \times \P^1$ &\ref{sec:all1s}
\end{tabular}
\end{center}

\subsection{$6=6$}\label{sec:trivial}
	This case corresponds to the trivial case $\L = \O_C$. 

\subsection{$6 = 4 + 2$}\label{sec:2plus4}
	In this case the degrees of $M$ are
		$$\deg(M) = \begin{pmatrix}
		4 & 3 \\
		3 & 2 \end{pmatrix}$$
	and the resolution of $\L$ has the form
		\[ 
			0 \rightarrow  \O(-5) \oplus \O(-4) \rightarrow \O(-1)\oplus \O(-2) 
			\rightarrow \L \rightarrow 0.
		\]
	The quadric bundle over $\P^2$ is simply $X \rightarrow \P^2$, the double cover of $\P^2$ ramfied on the sextic $C$.  Therefore the even Clifford algebra is  $\O_X$ and total Clifford algebra is an order with centre $\P^2$ and ramified on $C$.  In particular, $\calA_X(\calL)$ is trivial in $\Br X$.

	By Proposition~\ref{Ptangency}, the sextic is tangent to the conic and  quartic defined by the diagonal entries.  Since the conic is rational, it must split when pulled back to $X$.  Therefore $X$ has Picard rank at least $2$.  
	
	The line bundle $\L(1)$ has $h^0(\L(1))=1$ so we obtain an effective divisor $D$ of degree 6.  Since $2D = 2H,$ we see that $D$ is exactly the reduced locus of the intersection of $C$ with the conic.
	
	A parameter count shows that there is an $18$-dimensional space of sextic curves $C$ that are obtained this way. This gives a dense open subvariety of all sextics with a tangent conic or, equivalently, a dense open subvariety of all degree $2$ K3 surfaces whose intersection lattice contains the $2\times2$ block
	\[
		\left[
		\begin{array}{cc}
			2 &   2\\
			2 & -2
		\end{array}\right].
	\]
	
\subsection{$6=2+2+2$}\label{sec:all2s}
In this case the degrees of $M$ are:
$$\deg(M) = \begin{pmatrix}
2 & 2 & 2 \\
2 & 2 & 2 \\
2 & 2 & 2 \end{pmatrix}$$
and the resolution of $\L$ has the form 
$$ 0 \rightarrow \O(-4)^3 \rightarrow \O(-2)^3 \rightarrow \L \rightarrow 0.$$
The quadric bundle is a $(2,2)$ divisor in $\P^2 \times \P^2$.  The even Clifford algebra has rank 4 and is an order over $\P^2$ ramified on the sextic.
The total Clifford algebra has rank 8 and is an Azumaya algebra on the
double cover $X$.  The generic sextic will have $2^{20}-1 = \# \Pic C[2]-1$
such line bundles and thus $2^{20} - 1$ many ways to write the sextic as such a determinant.

\subsection{$6 = 5 + 1$}\label{sec:1plus5}
	In this case, the entries of the symmetric matrix have degrees
	$$ \deg M =\begin{pmatrix} 5 & 3 \\ 3 & 1 \end{pmatrix}.$$
	This gives rise to a resolution of $\L$ of the form
		\[
		 0 \rightarrow \OO_{\PP^2}(-5) \oplus \OO_{\PP^2}(-3)  
		 \stackrel{M}{\rightarrow} \OO_{\PP^2}\oplus \OO_{\PP^2}(-2) 
		 \rightarrow \calL \rightarrow 0
		\]
	so $\calL(1)$ is an odd theta characteristic.
	As in the $4 + 2$ case, the quadric bundle over $\P^2$ is simply $X \rightarrow \P^2$, the double cover of $\P^2$ ramfied on the sextic $C$.  Therefore the even Clifford algebra is  $\O_X$ and total Clifford algebra is an order with centre $\P^2$ and ramified on $C$.  In particular, $\calA_X(\calL)$ is trivial in $\Br X$.

	The sextic $C$ admits a tritangent, and a tangent quintic, again by Proposition~\ref{Ptangency}.  The tritangent line implies that $\rank \Pic X \geq 2$.  A parameter count shows that we obtain a dense open subvariety of all sextics with a tritangent line or, equivalently, a dense open subvariety of all degree $2$ K3 surfaces whose intersection lattice contains the $2\times 2$ block
	\[
		\left[\begin{array}{cc}
			2 & 1\\
			1 & -2
		\end{array}
		\right].
	\]

\subsection{$6=3+3$}\label{sec:3plus3}
	{
	Here the matrix has degrees 
	$$\deg M =\begin{pmatrix} 3 & 3 \\ 3 & 3 \end{pmatrix}.$$
	Therefore the resolution of $\L$ is of the form 
	$$ 0 \rightarrow \OO_{\PP^2}(-5)^2 \stackrel{M}{\rightarrow} \OO_{\PP^2}(-1)^2 \rightarrow \calL \rightarrow 0$$
	and $\calL(1)$ is an even theta characteristic.  As above, the quadric bundle is $X\to \PP^2$ and the Brauer class is trivial.
	
	Since the quadric bundle corresponding to $M$ is exactly $X$, we may view $X$ as a $(2,3)$ hypersurface in $\P^1 \times \P^2$.  These surfaces admit an elliptic fibration via the projection to $\P^1$.  This already implies that the Picard rank of $X$ is at least $2$, but we may also see it more explicitly.  Let $a,b,c$ be the upper triangular entries of $M$ so $C = V(ac - b^2)$.  Then consider the genus $1$ curve $Z = V(a)$.  Since $ac - b^2$ is a square modulo $a$, $Z$ splits when pulled back to $X$ thereby giving a divisor that is algebraically independent from the hyperplane class.
	
	A parameter count shows that we obtain a dense open subvariety of all degree $2$ K3 surfaces whose intersection lattice contains the $2\times 2$ block
	\[
		\left[\begin{array}{cc}
			2 & 3\\
			3 & 0
		\end{array}
		\right].
	\]
	Note that this is not a dense open subvariety of sextics that have a cubic which is everywhere tangent, since a tangent cubic is not enough to guarantee a jump in Picard rank.}

\subsection{$6 = 3 + 1 + 1 + 1$}\label{sec:cubic4fold}
	In this case 
	$M$ is a symmetric matrix with degrees 
	\[
		\deg M =
		\begin{pmatrix} 
			3 & 2 & 2 & 2 \\
			2 & 1 & 1 & 1 \\
			2 & 1 & 1 & 1 \\
			2 & 1 & 1 & 1
		\end{pmatrix}.
	\]
	Hence, we have a resolution of $\L$ of the form
	\[
		 0 \rightarrow \OO_{\PP^2}(-4) \oplus \OO_{\PP^2}(-3)^3  
		 \stackrel{M}{\rightarrow} \OO_{\PP^2}(-1)\oplus\OO_{\PP^2}(-2)^3  
		 \rightarrow \calL \rightarrow 0,
	 \]
	which shows that $\L(1)$ is an odd theta characteristic.

	In this case we naturally obtain a cubic fourfold as described in Example~\ref{ex:GenericPartitions}.  We can realize the K3 surface $X$ together with its double cover map to $\PP^2$ in the Stein factorization of the map from the maximal isotropic Grassmanian of the quadric bundle to the base.  Also, as before, we obtain a Brauer class on the K3 surface from the even Clifford algebra of the quadric bundle.

\subsection{$6=1+1+1+1+1+1$}\label{sec:all1s}
	Here the matrix $M$ is given by a $6 \times 6$ matrix of linear forms 
on $\PP^2$.  In this case we have a resolution of $\L$ of the form
\[
 0 \to \OO_{\PP^2}(-3)^6 \stackrel{M}{\rightarrow} \OO_{\PP^2}(-2)^6 \rightarrow \calL \rightarrow 0,
\]
thus $\calL(1)$ is an even theta characteristic.

As explained in Example~\ref{ex:GenericPartitions}, we obtain a $(2,1)$ hypersurface in $\PP^5\times\PP^2$, and projection on the second factor gives us a quadric bundle over $\PP^2$.  Here, as above, we can realize the K3 surface $X$ together with its double cover map to $\PP^2$ in the Stein factorization of the map from the maximal isotropic Grassmanian of the quadric bundle to the base $\PP^2$.

The $(2,1)$ hypersurface in $\PP^5\times\PP^2$ can also be thought of as a net of quadrics in $\PP^5$; the base locus of this net is a degree $8$ K3 surface $Y$.  It is well known (see, for example,~\cite{KuzQuad}) that $Y$ is derived equivalent to $X$ twisted by the Brauer class obtained from $M$.

	\appendix
	 \section{Combinatorics of Proposition \ref{Pdimcount} (by Hugh Thomas)}

 We complete the proof of Proposition \ref{Pdimcount} by showing that 
 the bound on the dimension of the moduli space of 
 sheaves $\mathcal L$ given by Proposition \ref{Pdimcount} is always at most the dimension of the
 moduli space of plane curves, up to projective isomorphism.  Furthermore, we characterize the cases of equality.  
 This all follows form the following combinatorial proposition.
 \begin{prop}\label{Pcombinatorics}

 Let $e$ be a positive integer, and let $d_1 + \cdots + d_n$ be a partition of $e$ such that the $d_i$ are all positive integers with the same parity, 
 forming a weakly decreasing sequence.  Let $a_i = (e - d_i - \e)/2$, where $\e$ is either $0$ or $1$, chosen so that 
 $a_i$ is an integer.  Then
 \[
  \sum_{1 \leq i \leq j \leq n} \binom{e-a_i-a_j-\e+2}{2} - \sum_{1
   \leq i , j \leq n} \binom{a_i-a_j+2}{2} -8 \leq \binom{e+2}{2} - 9.
   \]
 Furthermore, equality holds exactly when the partition $d_1 + \cdots + d_n$ is of the form
 $1 + \cdots + 1$, $2 + \cdots + 2$, $3 + 1 + \cdots + 1$, or simply $e$.
 \end{prop}

 \begin{proof}
 We observe that 

 $${e-a_i-a_j-\varepsilon+2 \choose 2} = \binom{\frac{d_i+d_j}{2} +2}{2}$$

 and 

 $$\binom {a_i - a_j +2}{2} = \binom{\frac{d_j-d_i}{2} +2}{2}$$

 We consider the contribution to the first sum on the left hand side in Proposition \ref{Pcombinatorics}
 from a fixed pair $i,j$ 
 (with $i\leq j$), and
 to the second sum from $i,j$ and, if $j\ne i$, also  $j,i$.  

 We note that typically only one of the two contributions to the second
 sum which we are presently
 considering is non-zero.  The one time when there are two non-zero contributions to
 the second sum is when $d_i=d_j$ but $i\ne j$.  In this case, both
 contributions are 1.  Let us write $f$ for the partition of $n$ whose
 parts record the number of times parts appear in $d$.  The number
 of times we obtain two contributions to the second sum is then
 $\sum \binom{ f_i}{2}$.
 We can therefore say:

 $$\sum_{i\leq j} {e-a_i-a_j-\varepsilon+2 \choose 2}-\sum_{i,j} \binom {a_i - a_j +2}{2} = \sum _{i\leq j} 
 \left(\binom{\frac{d_i+d_j}{2} +2}{2} - \binom{\frac{d_i-d_j}{2} +2}{2}\right) -
 \sum_i \binom{f_i}{2}$$

 The difference of binomial coefficients can be expanded as $(d_id_j+3d_j)/2$.  
 Therefore, we would like to prove that the following is non-negative, and to 
 characterize the cases when it equals zero:
 $$ \frac {e^2 +3e}{2} - \sum_{i\leq j} \frac{d_id_j +3d_j}{2} +\sum_i 
 \binom{f_i}{2}$$

 Expand $e$ as $\sum_i d_i$, and $e^2$ as $\sum d_i^2 + \sum_{i<j} 2d_id_j$.
 Split up the sum over $i\leq j$ in the previous expression into terms with $i$ and $j$
 equal, and terms with $i<j$.    
 The sum with $i=j$ cancels nicely with 
 elements from the expansion of the expression in $e$.  We are left with:

 $$\sum_{i<j} \frac{d_id_j -3d_j}{2} + \sum_i \binom{f_i}{2}$$ 

 We can rewrite this as:

 $$\sum_i \binom{f_i}{2} + \sum_{j} \frac{ (\sum_{i<j} d_i -3(j-1))d_j}{2} $$

 In the second sum, consider the contribution for a particular $j$ value. 
 We note that $(\sum_{i<j} d_i) - 3(j-1)$ is negative if the average over 
 the $d_i$ with $i<j$ is less than 3.  In order for this to happen,
 $d_{j-1}$ must be less than 3.  Let $j_0$ be the smallest position for
 which the contribution is strictly negative.  

 Suppose we are in the case that the parts of the partition are all even.  
 Then $d_{j_0-1}$ must be actually 2, and all the subsequent $d_j$ with
 $j_0-1\leq j \leq n$ must also be 2.  By our assumptions, the contribution
 for $j=j_0$ is $-1$, and for each successive $j$ with 
 $j_0\leq j \leq n$, we find that the contribution is $-(j-j_0 +1)$.
 However, we have at least $n-j+2$ parts which are all equal to 2.  Thus
 they contribute $\binom{n-j+2}{2}$ to $\sum \binom{f_i}{2}$, which exactly
 cancels out this total contribution.  The result is therefore non-negative.  
 In order for it to be zero, we must never have had a positive contribution
 to the second sum, which means the running average of the $d_i$'s must never
 have been above 3, except in the case that the partition has
 only one part.  Thus equality implies that $d_i=2$ for all $i$, or the partition
 has only one part.  

 Suppose we are in the case where the parts of the partition are all odd.
 Then $d_{j_0-1}$ must be actually 1, and all the subsequent $d_{t}$ with
 $j_0-1\leq j \leq n$ must also be 1.  The analysis goes through in exactly
 the same way, and we obtain non-negativity.  In the case of equality, we 
 conclude that the running average of the $d_i$'s
 must never have been more than 3, except in the case that the parition has
 only one part.  This means that, if the partition has more than one part,
 it must consist
 of a number of 3's, followed by a number of 1's.  However, if there were
 more than one 3, then this would contribute positively to $\sum \binom{f_i}{2}$,
 forcing the sum to be positive.  We are reduced to the cases that $d_i=1$
 for all $i$, and $d_1=3$, $d_i=1$ for $i>1$, plus the case that the partition has only one part.
 \end{proof}

\end{document}